\documentclass[11pt]{amsart}
\usepackage[margin=1.4in]{geometry}
\usepackage{amsmath,amsfonts,amssymb,amsthm,mathtools}
\usepackage{dsfont}
\usepackage{cases}
\usepackage{url}
\usepackage[all]{xy}
\usepackage{tikz-cd}
\usepackage{array}
\usepackage{mathtools}
\usepackage{faktor}
\usepackage{caption} 
\usepackage{verbatim,eucal}
\usepackage{xfrac}
\usepackage{MnSymbol}
\usepackage{multicol}
\usepackage{longtable}
\usepackage{mleftright}

\usepackage{enumitem}
\usepackage{tikz-cd}
\usepackage{xparse}
\usepackage{tikz}
\usepackage{pgfkeys}
\usetikzlibrary{matrix,backgrounds}
\pgfdeclarelayer{myback}
\pgfsetlayers{myback,background,main}

\tikzset{mycolor/.style = {line width=1bp,color=#1}}%
\tikzset{myfillcolor/.style = {draw,fill=#1}}%

\NewDocumentCommand{\highlight}{O{blue!40} m m}{%
\draw[mycolor=#1] (#2.north west)rectangle (#3.south east);
}

\NewDocumentCommand{\fhighlight}{O{blue!40} m m}{%
\draw[myfillcolor=#1] (#2.north west)rectangle (#3.south east);
}

\usepackage{letltxmacro}
\LetLtxMacro\Oldfootnote\footnote

\usepackage{hyperref}
\hypersetup{colorlinks=true,citecolor=blue!70!black,linkcolor=red!60!black,linktocpage=true}
\usepackage[capitalise,nameinlink]{cleveref}
\crefname{enumi}{part}{parts}

\makeatother
\numberwithin{equation}{section}
\newtheorem{thm}[equation]{Theorem} 
\newtheorem{prop}[equation]{Proposition}
\crefname{prop}{Proposition}{Propositions}
\newtheorem{lemma}[equation]{Lemma} 
\newtheorem{cor}[equation]{Corollary}
\crefname{cor}{Corollary}{Corollaries}
\newtheorem{example}[equation]{Example}
\newtheorem{remark}[equation]{Remark}
\newtheorem{definition}[equation]{Definition}

\DeclareMathOperator{\gr}{gr}

\DeclareMathOperator{\Span}{Span}

\DeclareMathOperator{\lcm}{lcm}
\DeclareMathOperator{\Aut}{Aut}
\DeclareMathOperator{\Stab}{{\rm Stab}}

\newcommand{\uni}{\rm{uni}}
\newcommand{\p}{ \mathcal{B} }

\newcommand{\q}{ {\mathfrak{q} } }
\newcommand{\Mat}{ { \rm{Mat} } }
\newcommand{\GL}{ { \rm{GL} } }

\renewcommand{\k}{{\mathbb K}}

\newcommand{\Q}{{\mathbb Q}}

\newcommand{\Hom}{\mbox{\rm Hom\,}}

\renewcommand{\ker}{\mbox{\rm Ker\,}}

\newcommand{\ot}{\otimes}

\newcommand{\CC}{\mathbb{C}}

\newenvironment{smallpmatrix}
  {\left(\begin{smallmatrix}}
  {\end{smallmatrix}\right)}

\begin{document}


\begin{abstract}
We examine the graded automorphism groups
of quantum affine spaces
and classify these groups
for spaces of dimension $7$ or less.
Using permutation actions on partitions,
we investigate cases 
when the group decomposes as a product of graded automorphism groups 
of smaller dimensional spaces, and we
describe the 
groups
arising from the Kronecker tensor product of  independent quantum parameter matrices.  
\end{abstract}

\title[Stabilizing Automorphisms of Quantum Affine Space]
{Stabilizing Automorphisms of Quantum Affine Space}

\date{July 14, 2026}

\author{Ethan Jensen$^{1}$}
\address{Department of Mathematics, University of North Texas,
Denton, Texas 76203, USA}
\email{ethan.jensen@unt.edu}

\author{Anne V.\ Shepler$^{2}$}
\address{Department of Mathematics, University of North Texas,
Denton, Texas 76203, USA}
\email{ashepler@unt.edu}


\thanks{1.\ Dept.\ of Mathematics,
  University of North Texas, 
  Denton, Texas, ethan.jensen@my.unt.edu}
\thanks{2.\ Corresponding author, Dept.\ of Mathematics,
  University of North Texas, Denton, Texas, ashepler@unt.edu}

\maketitle

\section{Introduction}
Determining the automorphism groups of algebras
remains a challenging task,
with difficulties 
 even in the case of commutative polynomial rings over fields.
Consider, for example,
Nagata's 1972 wildness conjecture
on the automorphism group of
$\CC[x_1, x_2, x_3]$ 
proved in 2004 by
Shestakov and Umirbaev 
\cite{shestakov2004tame},
see also Kraft~\cite{Kraft}.
Recent attention has turned to automorphism groups of noncommutative algebras viewed as coordinate rings.
We investigate the graded automorphism 
group of quantum affine spaces.
These 
are
finitely generated algebras
with each pair
of generators commuting up to 
a nonzero scalar.

The graded automorphisms
of a quantum affine space
$S_{\q}(V)$
give critical information
on the ungraded automorphisms.
Here $V\cong \k^n$ is the vector space
spanned by generators 
over a field $\k$
and 
$\q$ is a matrix 
of quantum scalars
recording the noncommutative multiplication.
Often the group of all
automorphisms
$\Aut(S_{\q}(V))$
coincides
with the group of graded automorphisms 
$\Aut_{\gr}(S_{\q}(V))$,
see Ceken, Palmieri, Wang, and Zhang~\cite[Section~3]{Ceken2016}
and
Yakimov~\cite[Corollary~3.7]{Yakimov2014}.
More generally, 
$\Aut(S_\q(V))$
is the semidirect product
of the unipotent automorphisms 
$\Aut_{\uni}(S_{\q}(V))$ with $\Aut_{\gr}(S_{\q}(V))$
unless $\q$
contains a row of all $1$'s,
see 
\cite[Lemma 3.2]{Ceken2016}.
The  groups 
$\Aut_{\gr}(S_{\q}(V))$
have been classified in
low dimensions, see
see Alev and Dumas~\cite{alev1996rigidite} 
 and Levandovskyy and Shepler~\cite{LevandovskyyShepler}.
 Explicit classifications in higher dimensions
have been lacking.
See also
 Alev and Chamarie~\cite{AlC}
and  Artamonov and  Wisbauer \cite{artamonov2001homological}.

We write the quantum affine space
as
$S_{\q}(V)=\k_\q[v_1, \ldots, v_n]$
generated by $v_1, \ldots, v_n$
with relations 
$v_j v_i = q_{ij}\, v_i v_j$ 
for $q_{ij}$
in $\k^*$ with $q_{ij}q_{ji}=1=q_{ii}$
with matrix of quantum scalars
$\q=\{ q_{ij} \}$,
also known as the skew polynomial ring.
We use the terms {\em quantum affine space} and {\em skew polynomial algebra} interchangeably; they both refer to the
algebra $S_{\q}(V)$.
Various results give the graded automorphism group in
special cases,
for example,
$\Aut_{\gr}(S_{\q}(V))
= \Aut(S_{\q}(V))
\cong (\k^\times)^n$
when the $q_{ij}$ are generic,
see 
\cite{AlC} and 
\cite{Yakimov2014}.
See \cite{AlC}
and
\cite{Ceken2016}
for the cases at the other extreme when
the quantum scalars $q_{ij}$
agree and/or are all $\pm 1$.
Bazlov and Berenstein \cite{bazlov2009noncommutative}
described
\(\Aut_{\gr}(S_{\q}(V))\)
as a product 
of subgroups 
of $\GL(V)$ with nontrivial overlap.

Jin~\cite{Jin} recently gave an elegant theorem describing  \(\Aut_{\gr}(S_{\q}(V))\)
as a semidirect product of general linear groups with a quotient of permutation groups.
We give a short and direct proof of a slightly different
description using the Splitting Lemma.
This characterization aids in a 
classification of these groups up to \(n=\dim V \leq 7\).
We give the semidirect product
structure explicitly
in terms of 
a maximal subgroup whose orbits
refine a particular partition
of 
\([n]^2 = \{(i,j): 1 \leq i,j \leq n\}\). 
Graded
automorphisms must permute
blocks of 
$\{1, \ldots, n\}$ of the same size, 
where each block indexes
identical rows of $\q$.
We leverage the preservation of block size to obtain an efficient algorithm (computationally
inexpensive) used in the classification.

We use the semidirect product structure to decompose 
\(\Aut_{\gr}(S_\q(V))\) 
as a direct product of quantum affine spaces
of potentially lower dimension
(i.e., fewer generators)
in \cref{DirectProductDecomp}
(see also \cref{ConstantOffDiagonalBlocks}).
This decomposition implies that the set of groups
\(\Aut_{\gr}(S_\q(V))\) for all $\q$ and $V$
is closed under direct products, 
see \cref{ClosedUnderDirectProducts}.
In addition, we give the graded automorphisms of 
 \(S_{\q \ot \q'}(V \ot V')\)
formed from the Kronecker product of two quantum parameter
matrices $\q$ and $\q'$
sufficiently independent
in \cref{TensorProductAutos}.

One asks which groups may arise as the graded automorphism group
of a quantum affine space.
This is determined jointly by 
the base field $\k$ and the diagonal action of 
subgroups of the symmetric group \(\mathfrak{S}_n\) on $[n]^2$.
With small assumptions on the cardinality and characteristic of $\k$, we identify several infinite families of groups that appear, namely, all groups of the form
$(\k^\times)^n \rtimes G$
for $G$ the symmetric group $\mathfrak{S}_n$,
the dihedral group $D_{2n}$,
or a cyclic group.
Every graded automorphism group is either monomial
or 
determined by monomial groups acting in lower dimension,
see \cref{EarlierTable}, so such families
are helpful for classification results.
We also explain why the stabilizing permutation group
of a graded automorphism group must appear
infinitely often in any classification,
see \cref{InfiniteFamilyHugeField,BlockPermutationTypeAppearsInfinitelyOften}.

\vspace{1ex}

\subsection*{Outline}
In \cref{SemiDirectSection}, we consider the structure
of the group of graded automorphisms
of skew polynomial rings.
We identify in \cref{InfiniteFamiliesSection} certain countably infinite families that appear as such groups. 
In \cref{DirectProductSection},
we decompose \(\Aut_{\gr}(S_\q(V))\)  
as a direct product of 
graded automorphism groups of subalgebras of $S_{\q}(V)$
generated by fewer variables,
and in 
\cref{TensorProductSection},
we consider skew polynomial rings
arising from the Kronecker product
(tensor product) of quantum parameter
matrices.
Lastly, we classify
in \cref{ClassificationSection}
the groups \(\Aut_{\gr}(S_\q(V))\)  
for \(\dim V \leq 7\).


\vspace{1ex}

\subsection*{Conventions and notation}
We fix a field $\k$ throughout
of arbitrary characteristic.
By a partition, we mean set partition
unless otherwise indicated,
and we write
$[n]$ for $\{ 1,\ldots, n \}$.
For $V\cong \k^n$ with fixed basis \(\{v_1, \ldots, v_n\}\) 
and \(B \subset [n]\), 
we write \(V_B = \Span_\k\{v_i: i \in B\}\).


\section{Quantum affine space and graded automorphisms}
\label{SemiDirectSection}

We give here different descriptions
for the graded automorphism
group of quantum affine space.

\subsection*{Quantum affine space}
 
An $n\times n$ matrix $\q=\{ q_{ij}\}$ with entries
in $\k$ is a {\em system of quantum parameters} or a {\em quantum parameter matrix}
when
    $$
    q_{ii} = q_{ij}\, q_{ji} = 1
    \quad\text{ for } 1\leq i, j\leq n
    \, .
    $$
We fix a $n\times n$ quantum parameter matrix $\q$ and
a finite dimensional vector space $V$ over \(\k\) with basis $v_1, \ldots , v_n $.
The {\em skew polynomial algebra} 
\(S_\q(V)\)
(also called the {\em quantum polynomial ring} or {\em quantum affine space}) 
associated to \(\q\)
is the $\k$-algebra generated by
$v_1, \ldots, v_n$ 
with relations \(v_jv_i = q_{ij}v_iv_j\):
\[
S_\q(V) = \k\langle v_1, \dots, v_n \rangle\ /\ 
(v_jv_i - q_{ij}v_iv_j:\ 1 \leq i,j \leq n )
\, .
\]
We view $S_\q(V)$ as a graded algebra with $\deg v_i=1$
for all $i$, see \cite{GoodearlWarfield},
and write
$\Aut_{\gr}( S_{\q}(V) )$
for the ring of graded automorphisms
of $S_{\q}(V)$, i.e., automorphisms
that preserve degree.
Every graded automorphism
defines a linear transformation
on $V$ by restriction, 
and we view $\Aut_{\gr}( S_{\q}(V) )$
as a subgroup of $\GL(n, \k)$.
Conversely, \cref{IsoConditionOnIndices} below
describes which elements of
$\GL(n, \k)$ extend to graded automorphisms
on $S_{\q}(V)$
with respect to the basis $v_1, \ldots, v_n$,
 see also
\cite{bazlov2009noncommutative}
and \cite{Jin}.
We write $h=(h_{ij})$ in $\GL(n, \k)$
for $h v_j= \sum_j h_{ij} v_i$.

\begin{lemma}
\cite[Lemma 3.2]{LevandovskyyShepler} 
\label{IsoConditionOnIndices}
A matrix
$h\in \GL(n, \k)$
lies in $\Aut_{\gr}(S_\q(V))$ if and only if
    $$
        h_{i\ell}\, h_{jm} = 0 
    \quad\text{ or }\quad 
    q_{ij} = q_{\ell m} 
    \quad\text{ for all }\ \ 1\leq i,j,m,\ell\leq n
    \, .
    $$
\end{lemma}

\vspace{1ex}

\subsection*{Monomial automorphisms}  
    Let \(\mathbb{G}_n\) denote the group of monomial matrices 
    in $\GL(n, \k)$, i.e., invertible matrices 
    with exactly \(n\) nonzero entries. 
    Then \(\mathbb{G}_n\) 
    admits a decomposition
    $(\k^\times)^n \rtimes \mathfrak{S}_n$
    for \((\k^\times)^n\) identified with the diagonal matrices in \(\GL(n, \k)\)
    and the symmetric group \(\mathfrak{S}_n\) identified with the permutation matrices
    acting on \(V\) by permutation of basis
    elements $v_i$.
    The group \(\mathfrak{S}_n\) 
    acts on the set of \(n \times n\) 
    quantum parameter matrices by $\sigma: \q  \mapsto \q'$ 
    where \( q'_{ij} = q_{\sigma(i)\,\sigma(j)}\).
    A monomial matrix
    $h = d\, \sigma$
    for
    $d$ in $\GL(n, \k)$ diagonal and $\sigma$ in $\mathfrak{S}_n$
    lies in $\Aut_{\gr}(S_\q(V))$ exactly when 
    $\q$ is invariant under this action of $\sigma$,
     in which
    case we call $h$ a {\em monomial automorphism}.

\vspace{1ex}

\begin{remark}
\label{WhenMonomialGroupsArise}
{\em
Note that    \(\Aut_{\gr}(S_\q(V))\) contains only monomial automorphisms if and only if \(\q\) has no identical rows, see \cite[Lemma 3.5e]{KKZ}.
This implies that
if $q_{ij}\neq 1$ for all $i\neq j$,
then $\Aut_{\gr}(S_\q(V))$
is a monomial group of matrices,
see \cite[Lemma 3.4]{SheplerUhl}, \cite[Lemma 3.2]{Ceken2016}, and \cite[Prop 3.9]{Ceken2016}.
    }
\end{remark}

\vspace{1ex}

\begin{remark}\label{SemiDirectProducts}
{\em 
We will use the semi-direct product structure
arising from the Splitting Lemma:
If $\pi: G\rightarrow G'$ and
$\iota:G'\rightarrow G$ are group homomorphisms
with $\pi\iota=\text{Id}$,
then the exact sequence
$1\rightarrow \ker\pi \hookrightarrow
G \xrightarrow{ \pi } G' \rightarrow 1$
of groups splits and 
$$G\ \cong\ \ker \pi \rtimes G'$$
under the map $g \mapsto (g \, \iota \pi (g^{-1}), \pi(g))$
with multiplication
$
    (g, \sigma)(h, \tau) 
    =
    \big(g\, \iota(\sigma)\, h\, \iota(\sigma)^{-1},\ \sigma\tau\big)
$.
}
\end{remark}

\vspace{1ex}

\subsection*{Structure of Graded Automorphism Group}
 Jin~\cite{Jin} elegantly presented 
 the graded automorphism group
 $\Aut_{\gr}(S_\q(V))$
as a
semidirect product.
We 
 give a short proof
of an equivalent structure
using the Splitting Lemma for groups
(see~\cref{SemiDirectProducts}).
We use this approach to
decompose the group of graded automorphisms in
terms of potentially smaller graded automorphism groups,
see
\cref{ConstantOffDiagonalBlocks}, \cref{DirectProductDecomp},
and \cref{TensorProductAutos}.
Our formulation allows us to
directly
leverage the fact that
the acting permutation group
permutes blocks of the same size,
making it easy to express
as a subgroup of a product of
symmetric groups 
read off immediately from the matrix $\q$.
Using this fact, we recharacterize 
the graded automorphism group in terms of maximal subgroups of a permutation
group with an orbit condition.

The graded automorphism group is a semidirect product of two types
of automorphisms:
The first preserves the subspaces of $V$ 
spanned by basis elements 
corresponding to identical rows in $\q$
and the second permutes these subspaces. 
We fix an \(n \times n\) quantum parameter matrix \(\q\)
defining $S_\q(V)$
and make heavy use of
a partition from \cite[Definition 3.1]{KKZ}:

\vspace{1ex}

\begin{definition}
{\em
    Let \(\p_\q\) be the partition of \([n] = \{1, \ldots ,n\}\) with \(i \sim j\) if
    the $i$-th and $j$-th rows of \(\q\) are identical, i.e., \(q_{im} = q_{jm}\) for all \(m\),
    in which case
    we say $i$ and $j$ {\em lie in the same block
    of $\q$}.
    }
\end{definition}

\vspace{1ex}

We decompose \(V\) as a direct sum $V = \bigoplus_{B \in \p_\q}V_B$ for
\(V_B = \Span\{v_i: i \in B\}\).
We identify $\GL(V_B)$
with the subgroup of $\GL(V)$
that acts via $\GL(V_B)$ on $V_B$ 
and fixes basis elements
$v_i$ with $i$ not in $B$.
Kirkman, Kuzmanovich, and Zhang \cite[Lemma 3.2]{KKZ} show that 
the groups \(\GL(V_B)\) are subgroups of \(\Aut_{\gr}(S_\q(V))\)
for all blocks \(B \in \p_\q\).
%
For \(B,C \subset [n]\), we
     consider the  minor matrix of size $|B|\times |C|$
     \begin{equation}\label{BlockMinorDefn}
              \q_{_{BC}} = (q_{ij})_{i \in B,\, j \in C}\, . 
 \end{equation}
Note that this minor has all entries identical 
whenever $B$ and $C$ are blocks of $\q$
and
 that $\q_{_{BB}}$ is a quantum parameter matrix for any \(B \subset [n]\). 
For $r=|\p_\q|$,
we identify the symmetric group \(\mathfrak{S}_r\)
with the group of permutations of
 the blocks of $\q$ so that $\mathfrak{S}_r$
acts on $\p_\q$.

For $\sigma$ in $ \mathfrak{S}_r$ preserving block size,
i.e., with $|\sigma(B)| = |B|$
for all blocks $B$ of $\q$,
let $\sigma \q$ be the $n\times n$
matrix whose minor matrices are given by
(compare with \cite{bazlov2009noncommutative})
\[
    (\sigma \q)_{_{BC}} = \q_{_{\sigma(B)\, \sigma(C)}}\ \quad\text{ for all } B,C \in \mathcal{\p_\q}
    \, .
\]
As $\q$ is a quantum parameter matrix, 
so is $\sigma\q$.
We consider the subgroup of permutations
that fix $\q$:
\begin{definition}
\label{BlockPermutationType}{\em 
The {\em stabilizing permutation group}
of $\Aut_{\gr}(S_\q(V))$
   is the group of permutations
    $$\begin{aligned}
    \Stab (\q)  
     &=
    \{\sigma \in \mathfrak{S}_r:\ 
    \sigma\text{ preserves block size and } 
    \sigma \q = \q \}\\
    & = 
    \{\sigma \in \mathfrak{S}_r:\ 
    \q_{_{B C}} 
    = 
    \q_{_{\sigma(B)\, \sigma(C)}}
    \text{ for all blocks $B,C$ of }
    \q \} \, .
       \end{aligned}
$$
}
\end{definition}

We view $\Stab(\q)$
as the {\em stabilizing automorphisms}
of quantum affine space under an injection 
recorded with the
next lemma.

\begin{lemma}\label{BlockAutoInjection}
    There exists an injective group homomorphism
    $\iota: \Stab (\q) \rightarrow \Aut_{\gr}(S_\q(V))
    $
    with image
   \(\prod_{B \in \p_\q}\Hom(V_B, V_{\sigma(B)})\).
 \end{lemma}
\begin{proof}
    For $\sigma$ in $\text{Stab}(\q)$,
   let
    \( \iota(\sigma)\) be the unique invertible linear map on $V$ which
    permutes the basis elements $v_i$ of $V$, sends $V_B$ to $V_{\sigma(B)}$ for each block \(B\) of \(\q\), and preserves the order 
    $v_1< v_2< \cdots < v_n$ 
    within each block, using 
    \(|B| = |\sigma(B)|\), so that
    \(\iota(\sigma)\) lies in \(\prod_{B \in \p_\q}\Hom(V_B, V_{\sigma(B)})\). 
    \cref{IsoConditionOnIndices} implies that \(\iota(\sigma)\) lies in \(\Aut_{\gr}(S_\q(V))\)
    since if
    \(\iota(\sigma)(v_\ell) = v_i\)
and \(\iota(\sigma)(v_m) = v_j\),
then
    \(q_{\ell m} = q_{i j}\)   
    as $ q_{_{BC}} = q_{_{\sigma(B)\sigma(C)}}
    $
    for  $\ell,m$ in respective blocks $B,C$ of $\q $.
    Finally, to see that \(\iota \) is a homomorphism,
    consider $\sigma'$ in $\text{Stab}(\q)$.
    The composition $\iota(\sigma)\iota(\sigma')$ 
    also permutes basis vectors of $V$ and preserves the order in each block,
    so by uniqueness, it agrees with
    $\iota(\sigma\sigma')$.
     \end{proof}

\begin{prop}\label{BlockAutoProjection}
    There exists a surjective group homomorphism \(\pi: \Aut_{\gr}(S_\q(V)) \rightarrow \Stab (\q)\)
    with 
    $\ker\pi = \prod_{B\in \p_\q}\GL(V_B) \)
    and $\pi\, \iota = \text{Id}$
    for $\iota$ the map of \cref{BlockAutoInjection}.
\end{prop}
\begin{proof}
For \(h =( h_{ij} ) \in \Aut_{\gr}(S_\q(V))\), 
define a function 
    $\pi(h): \p_\q \rightarrow \p_\q$ 
    which permutes blocks by
    $$ \pi(h)(B) = \{i:\, h_{ij} \neq 0 \textup{ for some } j \in B\}
    \, .
    $$
We argue $\pi(h)$ has the advertised codomain.
First note that $\pi(h)(B)$ is a subset
of a block for each block $B$.
Indeed, for $i$ and $\ell$ in $\pi(h)(B)$,
$h_{ij}\neq 0 \neq h_{\ell m}$
for some $j,m\in B$.
To see that $i$ and $\ell$ lie in the same block, 
fix an index $s$ and find $t$ with $h_{st}\neq 0$.
By \cref{IsoConditionOnIndices},
$q_{is}= q_{jt}$ 
as $h_{ij} h_{st}\neq 0$,
$q_{jt} = q_{mt}$ as $j$ and $m$ lie in the same block,
and $q_{mt} = q_{\ell s}$
as $h_{\ell m} h_{st} \neq 0$.
A similar argument shows that
$\pi(h)(B)$ is an entire block.
Indeed, if $\pi(h)(B)\subset C$ for $C$ a block
and $\ell\in C$, take
$m$ with $h_{\ell m}\neq 0$ and
$i\in \pi(h)(B)$ so that $h_{ij}\neq 0$
for some $j\in B$. Then for all $t$,
and $h_{st}\neq 0$,
$q_{jt}=q_{is}=q_{\ell s}=q_{mt}$
(using $h_{ij}h_{st}
\neq 0 \neq 
h_{\ell m} h_{st}$)
and hence rows $j$ and $m$ of $\q$ are identical.
Thus $m\in B$ and $\ell \in \pi(h)(B)$.

Note that $\pi(h)$ takes each block of $\q$
 to a block of the same size.
Otherwise, some block would be sent to a smaller block, and, after reindexing,
$h$ would be a block matrix with all entries $0$ above and below a
non-square block
(minor submatrix)
 forcing $\det h=0$.
Next notice that for $i \in B'=\pi(h)(B)$
and $\ell \in C'=\pi(h)(C)$,
there exist $j \in B$ and $m \in C$ with $h_{ij} h_{\ell m}\neq 0$. Then $q_{i\ell} = q_{j m}$ by \cref{IsoConditionOnIndices}, so $\q_{_{BC}} = \q_{_{B' C'}}$.

    To see that 
$\pi: \Aut_{\gr}(S_\q(V))\rightarrow \Stab(\q)$, 
$h\mapsto \pi(h)$,
is a group homomorphism, fix
    $g, h$ in $\Aut_{\gr}(S_\q(V))$
    and 
    \(i \in \pi(gh)(B)\). 
    Then \((gh)_{ij} \neq 0\) for some \(j \in B\) and so 
    $g_{i\ell} \neq 0 \neq h_{\ell j}$
    for some $\ell$.
    Then $\ell \in \pi(h)(B)$ and so 
    $i \in \pi(g)(\pi(h)(B))$. Thus, \(\pi(gh)(B) = \pi(g)\pi(h)(B)\).
A quick check confirms that
$\pi \, \iota = \text{Id}$
and that $\pi$ has the kernel claimed.
\end{proof}

\vspace{1ex}

\subsection*{Semidirect product structure}
We obtain a semidirect product 
    by letting $\Stab (\q)$ act
    on $\prod_{B \in \p_\q}\GL(V_B)$
    by conjugation via the map of
    \cref{BlockAutoInjection}:
    For $g, h$ in $
    \prod_{B \in \p_\q}\GL(V_B)
    \) and $\sigma, \tau$ in $ \Stab (\q)$,   
    \[
    (g, \sigma)(h, \tau) 
    \ =\
    \big(g\, \iota(\sigma)\, h\, \iota(\sigma)^{-1},\ \sigma\tau\big)
    \qquad\text{ in }
    \qquad
    \Big(\prod_{B \in \p_\q}
    \GL(V_B)\Big) \rtimes \Stab (\q)
    \, .
    \]
The composition in the first coordinate indeed lies in \(\prod_B \GL(V_B)\): 
    For a fixed block \(B\) of $\q$, 
    \[V_B \xrightarrow{\, \iota(\sigma)^{-1}\, }
    V_{\sigma^{-1}(B)} \xrightarrow{\ h\ }
    V_{\sigma^{-1}(B)} 
    \xrightarrow{\ \iota(\sigma) \ } V_{B} 
    \xrightarrow{\ g\hphantom{x} } 
    V_B
    \, . 
    \]
We obtain
a short proof of an alternate formulation of
the main result of Jin~\cite{Jin}
using the Splitting Lemma.
We use this
approach 
to establish
\cref{ConstantOffDiagonalBlocks},
\cref{DirectProductDecomp},
and the classification in \cref{ClassificationSection}.

\begin{cor}\label{BlockMatrixDecomposition}
    For a skew polynomial algebra \(S_\q(V)\),
       $$
     \Aut_{\gr}(S_\q(V)) 
    \ \cong\ 
    \Big(\prod_{B \in \p_\q}
    \GL(V_B)\Big)\, \rtimes\, \Stab (\q)
    \, .
    $$
\end{cor}
\begin{proof}
\cref{SemiDirectProducts} 
with the maps \(\iota: \Stab (\q) \rightarrow \Aut_{\gr}(S_\q(V)\)  
    and 
    \(\pi: \Aut_{\gr}(S_\q(V)) \rightarrow \Stab (\q)\) 
    of \cref{BlockAutoInjection} and \cref{BlockAutoProjection}
   give the isomorphism
    $
    g \mapsto \big(g \cdot \iota\pi (g^{-1}),\ \pi(g) \big) $.
\end{proof}

\vspace{1ex}

\subsection*{Orbit Condition}
We now
characterize the graded automorphism group
of a skew polynomial algebra
in terms of subgroups
of the symmetric group maximal
with respect to an orbit condition.
The maximality condition explains why so few permutation groups
appear as factors of the semidirect product 
in \cref{BlockMatrixDecomposition}.
It also allows one to rule out groups
in a classification of
graded automorphism groups, see
\cref{ClassificationSection}.

Recall that $\p_\q$ is 
the partition of \([n]\) 
given by grouping identical rows of $\q$
and we fix $r=|\p_\q|$
as the number of distinct rows,
with $\mathfrak{S}_r$ acting on $\p_\q$
by permutations.
We use the diagonal action of
\( \mathfrak{S}_r \) on \([r]^2=[r]\times [r]\)
and show that the stabilizing permutation group \(\Stab (\q)\) 
is maximal among all subgroups $G$
of $\mathfrak{S}_r$ 
with the same orbits, i.e., with $[r]^2/G
= [r]^2/ \Stab (\q)$.

\begin{lemma}
\label{lattice}
    For any partition $\mathcal{P}$ of $[r]^2$, 
    the set of subgroups 
    $\{ G \subset \mathfrak{S}_r:     
[r]^2/G \text{ refines } \mathcal{P}\}$
forms a lattice under subgroup inclusion.
\end{lemma}
\begin{proof}
The orbits of the trivial group  on \([r]^2\) are the sets \(\{(i,j)\}\) giving a refinement of any partition of \([r]^2\), so the set of subgroups is nonempty.
If \(G\) and \(H\) are subgroups 
with $[r]^2/G$ and \([r]^2/H\) refining $\mathcal{P}$, then \([r]^2/\langle G, H \rangle\) 
and \([r]^2/ (G\cap H)$
also refine \(\mathcal{P}\). 
\end{proof}

We now give a description of the 
graded automorphism groups
useful for the
classification
in \cref{ClassificationSection}.
We write any partition $\lambda$ of $n$
as $(\lambda_1^{m_1} \ldots \, \lambda_s^{m_s})$ 
with $\lambda_1, \ldots, \lambda_s$ the distinct 
parts and $m_i$ the multiplicity of $\lambda_i$, i.e., number of parts 
 equal to $\lambda_i$.
 Recall the minor matrices defined by blocks of $\q$, see  \cref{BlockMinorDefn}.
\begin{thm}
\label{BlockAutosSubsetMaximal}
For a quantum parameter matrix $\q$ with
blocks $B_1, \ldots, B_r$,
let \(\mathcal{P}\) be the partition of \([r]^2\) given by 
\((i,j) \sim (\ell,m)\) 
if \( \q_{_{B_i\, B_j}} = \q_{_{B_{\ell}\, B_{m}}}\).
Then
    \[\Aut_{\gr}(S_\q(V)) 
    \ \cong \ 
    \Big(\prod_{i=1}^{s}\GL(\lambda_i, \k)^{m_i}\Big)
    \, \rtimes\, G
    \]
    for 
   some integer partition 
    $\lambda=(\lambda_1^{m_1} \ldots \, \lambda_s^{m_s})$ 
    of $\dim V$
    with 
    $\sum_i m_i = r$
    and 
    \(G\) the maximal subgroup of \(\mathfrak{S}_r\) such that 
the orbit space \( [r]^2 /G\) refines \(\mathcal{P}\). 
    Here, $G=\Stab(\q)$, and $G$ acts by a subgroup of
    $\prod_{i=1}^s \mathfrak{S}_{m_i}$
    where each  $\mathfrak{S}_{m_i}$ permutes
    the $m_i$ copies of $\GL(\lambda_i, \k)$.
    \end{thm}
\begin{proof}
We use \cref{lattice} and \cref{BlockMatrixDecomposition}. 
For \(\sigma \in \mathfrak{S}_r\), 
the partition 
given by the orbits \([r]^2/\langle \sigma \rangle\) refines \(\mathcal{P}\) exactly when \(\q_{_{BC}} = \q_{_{\sigma(B)\, \sigma(C)}}\) for all blocks \(B,C\) of \(\p_\q\). Thus,
    \[
    G=
    \Stab(\q)=\{ \sigma\in \mathfrak{S}_r:
\sigma \text{ preserves block size and } \sigma \q = \q    
\}
    \ = \ 
    \{\sigma \in \mathfrak{S}_r: [r]^2/\langle \sigma \rangle \textup{ refines } 
    \mathcal{P}\}\]
    is maximal among the groups with orbits
    refining \(\mathcal{P}\),
    giving the first statement.

     For each block \(B \in \p_q\), 
     we identify \(\GL(V_B)\) with \(\GL(|B|, \k)\)
    and define the partition $\lambda$ of $n$
    corresponding to the block sizes:
    \(\lambda = ({\lambda_1}^{m_1}\ \ldots\ 
    \lambda_s^{m_s})\),
    where the distinct parts \(\lambda_i\) give the 
    distinct sizes of blocks of $\q$ 
    and \(m_i\) is the number of blocks of size \(\lambda_i\).
    Since each \(\sigma\) in \(G\)
    preserves block size,  
      \( G\) must be a subgroup of
    $\prod_{i=1}^s\mathfrak{S}_{m_i}$
    up to conjugation in $\mathfrak{S}_r$.
      \end{proof}

\vspace{1ex}

\begin{remark}{\em 
    Two extreme cases of \cref{BlockAutosSubsetMaximal}
    arise.
When  $\q$ has no duplicate rows,
       all the blocks have size $1$
       giving a 
generalization of \cite[Proposition~3.9]{Ceken2016}:
         $$\Aut_{\gr}(S_\q(V)) \ \cong\  
     (\k^{\times})^n \, \rtimes\, \Stab(\q)
     \, . $$
In contrast, when the block sizes of \(\q\) are all distinct
     recorded by a partition \(\lambda\),
     $\Stab(\q)$ is trivial:
    $$ 
    \Aut_{\gr}(S_\q(V)) \ \cong\  \prod_{\lambda_i}\GL(\lambda_i, \k)
    \, .
    $$  
} 
\end{remark}

 \vspace{1ex}

We now consider some examples.

\begin{example}{\em
    Say \(|\k| > 3\) and let $ \q=\left(\begin{smallmatrix}
       A & B & C \\
       C & A & B \\
       B & C & A
       \end{smallmatrix}\right)
       \rule[-2ex]{0ex}{3ex}
$
for
   $
   A=\left(\begin{smallmatrix}
       1 & 1 \\
       1 & 1
       \end{smallmatrix}\right)$,
$   B=\left(\begin{smallmatrix}
       a & a \\
       a & a
       \end{smallmatrix}\right)$,
       and
 $   C=\left(\begin{smallmatrix}
       a^{-1} & a^{-1} \\
       a^{-1} & a^{-1}
       \end{smallmatrix}\right)$ 
for \(a \in \k\) with \(a \not\in \{0,1,-1\}\). 
Then \(\p_\q = \{\{1,2\}, \{3,4\}, \{5,6\}\}\)
and
\cref{BlockAutosSubsetMaximal} implies that \(\Aut_{\gr}(S_\q(V)) \cong D \rtimes\Stab (\q)\) where \(\Stab (\q) \subset \mathfrak{S}_3\) and 
    \[
    D = \Big\{\begin{smallpmatrix}
        M_1 & 0 & 0 \\
        0 & M_2 & 0 \\
        0 & 0 & M_3
    \end{smallpmatrix}:\ M_i \in \GL(2, \k)\Big\} 
    \ \cong\ 
    (\GL(2, \k))^3\, .
    \]
    A straightforward calculation shows that \(\Stab (\q) = \langle (1\, 2\, 3) \rangle\) and thus
    \[\Aut_{\gr}(S_\q(V)) = 
    \Big\{
    \begin{smallpmatrix}
        M_1 & 0 & 0 \\
        0 & M_2 & 0 \\
        0 & 0 & M_3
    \end{smallpmatrix},\ \begin{smallpmatrix}
        0 & 0 & M_3 \\
        M_1 & 0 & 0 \\
        0 & M_2 & 0
    \end{smallpmatrix},\ \begin{smallpmatrix}
        0 & M_2 & 0 \\
        0 & 0 & M_3 \\
        M_1 & 0 & 0
    \end{smallpmatrix}: 
    M_i \in \GL(2, \k)
    \Big\}\, .\]
}
\end{example}

\vspace{1ex}

\begin{example}{\em 
    Let \(V \cong \k^8\) with \(\Q \subset \k\) and let
    \[\q = \begin{smallpmatrix}
        1 & \frac{1}{2} & \frac{1}{2} & -1 & -1 & 1 & 3 & \frac{1}{2} \\
        2 & 1 & 1 & 2 & 2 & 2 & 4 & 1 \\
        2 & 1 & 1 & 2 & 2 & 2 & 4 & 1 \\
        -1 & \frac{1}{2} & \frac{1}{2} & 1 & 1 & -1 & 3 & \frac{1}{2} \\
        -1 & \frac{1}{2} & \frac{1}{2} & 1 & 1 & -1 & 3 & \frac{1}{2} \\
        1 & \frac{1}{2} & \frac{1}{2} & -1 & -1 & 1 & 3 & \frac{1}{2} \\
        \frac{1}{3} & \frac{1}{4} & \frac{1}{4} & \frac{1}{3} & \frac{1}{3} & \frac{1}{3} & 1 & \frac{1}{4} \\
        2 & 1 & 1 & 2 & 2 & 2 & 4 & 1 \\
        \end{smallpmatrix}\, .\]
Here,
    \(\p_\q = \{\{1,6\}, \{2,3,8\}, \{4,5\}, \{7\}\}\) and
     \cref{BlockMatrixDecomposition} implies that
    \[
    \Aut_{\gr}(S_{\q}(V)) = \left(\GL(V_{\{1,6\}}) \times \GL(V_{\{2,3,8\}})\times \GL(V_{\{4,5\}}) \times \GL(V_{\{7\}})\right)
    \, \rtimes\, \Stab (\q)
    \, .
    \]
    A quick computation verifies
    that \(\Stab (\q) = \langle (1\ 3) \rangle\), 
    which swaps rows $1$ and $6$ with $4$ and $5$. 
    Thus
    \[\Aut_{\gr}(S_{\q}(V)) = \GL(8, \k) \cap \left\{ \begin{smallpmatrix}
        * & & & & & * & & \\
         & * & * & & & & & * \\
         & * & * & & & & & * \\
         & & & * & * & & & \\
         & & & * & * & & & \\
         * & & & & & * & & \\
         & & & & & & * & \\
         & * & * & & & & & *
    \end{smallpmatrix},\ 
    \begin{smallpmatrix}
         & & & * & * & & & \\
         & * & * & & & & & * \\
         & * & * & & & & & * \\
        * & & & & & * & & \\
        * & & & & & * & & \\
         & & & * & * & & & \\
         & & & & & & * & \\
         & * & * & & & & & *
    \end{smallpmatrix}\right\}
    \, .
    \]
}
\end{example}

\vspace{1ex}

\begin{example}\label{TwoExamplesWithBlocks}
{\em
   For \(\k = \mathbb{F}_5$,
   $\dim V = 9$, 
   $
   A=\left(\begin{smallmatrix}
       1 & 2 & 3 \\
       3 & 1 & 2 \\
       2 & 3 & 1
       \end{smallmatrix}\right)$,
$   B=\left(\begin{smallmatrix}
       2 & 2 & 2 \\
       2 & 2 & 2 \\
       2 & 2 & 2
       \end{smallmatrix}\right)$,
 $   C=\left(\begin{smallmatrix}
       3 & 3 & 3 \\
       3 & 3 & 3 \\
       3 & 3 & 3
       \end{smallmatrix}\right)$,   
and
$ \q=\left(\begin{smallmatrix}
       A & B & C \\
       C & A & B \\
       B & C & A
       \end{smallmatrix}\right)
$,
$$
\Aut_{\gr}(S_\q(V)) 
\ \cong\ (\k^\times)^9 
\,\rtimes\, 
\big[\langle (1\, 2\, 3) \rangle ^3 
        \, \rtimes\, \langle (1\, 2\, 3) \rangle\big] 
        \ \cong\
        \big[(\k^\times)^3 
        \,\rtimes\, 
        \langle (1\, 2\, 3) \rangle \big]^3 \,\rtimes\, 
        \langle (1\, 2\, 3) \rangle
        \, .
        $$
}
\end{example}

\vspace{1ex}

\subsection*{Constant off-diagonal minors}
The last examples leads us to consider quantum parameter matrices with constant matrix minors.
Given a quantum parameter matrix $\q$
and a partition $\mathcal{P}$ of \([n]\)
with blocks $D$ and $E$,
we call the submatrix $\q_{_{DE}}$ a
{\em diagonal $\mathcal{P}$-minor} when $D= E$
and an
{\em off-diagonal $\mathcal{P}$-minor} when $D\neq E$.
Note that $\q$ is constant on
the off-diagonal minors
for the block-circle decomposition
partition, see \cite[Proposition~4.3b]{KKZ}.

In the next proposition,
we consider the induced action
of $\Stab(\q)$
on the set of partitions $\mathcal{P}$
of $[n]$ refined by $\p_\q$:
For $\sigma$ in $\Stab(\q)$,
set 
$\sigma\mathcal{P}$ to be the partition
of $[n]$ whose blocks are
$\sigma(B_1)\cup\cdots\cup\sigma(B_m)$ 
for $B_1\cup\cdots\cup B_m$ a block of 
 $\mathcal{P}$,
where the $B_i$ are blocks of $\p_\q$.

\begin{prop}
\label{ConstantOffDiagonalBlocks}
Suppose \(\q\) is an $n\times n$ 
    quantum parameter matrix 
    constant on off-diagonal
$\mathcal{P}$-minors for a partition 
    \(   \mathcal{P} 
    \)  of \([n]\) 
    refined by \(\p_\q\).
    If  $\Stab(\q)$ fixes $\mathcal{P}$,
    then 
    \[\Aut_{\gr}(S_\q(V)) 
    \ \cong\  
    \prod_{D \in \mathcal{P}}\Aut_{\gr}(S_{\q_{_{DD}}}(V_D))\, \rtimes\, \Stab_{\, \mathcal{P}}(\q)
    \, 
    \]
for $\Stab_{\, \mathcal{P}}(\q)=\{ \sigma \in \mathfrak{S}_{|\mathcal{P}|}:
|D| = |\sigma(D)| \text{ and }\q_{_{\sigma(D)\, \sigma(E)}} = \q_{_{DE}} 
\text{ for all blocks 
$D,E$
of $\mathcal{P}$}
\}$.
\end{prop}
\begin{proof}
We first define a group homomorphism 
$\psi: \Stab(\q)\rightarrow 
\Stab_{\, \mathcal{P}}(\q)$.
Fix $\sigma$ in $\Stab(\q)$.
As $\sigma \mathcal{P}=\mathcal{P}$,
for any block
$D = B_1 \cup \cdots \cup B_m$ of $\mathcal{P}$
with \(B_i \in \p_\q\), 
 \[
\sigma(D) 
= \sigma(B_1) \cup \cdots \cup \sigma(B_m)
 \, 
\]
is again a block of $\mathcal{P}$ 
with \(|\sigma(D)| = |D|\) 
as $\sigma$ permutes the blocks of $\q$ of the same size.
We may thus identify $\sigma$ with an element of
$\mathfrak{S}_{|\mathcal{P}|}$.
Since \(\q_{_{BC}} = \q_{_{\sigma(B)\sigma(C)}}\)
for all blocks $B$, $C$ of $\q$,
 it follows that \(\q_{_{DE}} = \q_{\sigma(D)\sigma(E)}\)
for all blocks $D$, $E$ of $\mathcal{P}$
and $\sigma$ lies in 
$\Stab_{\, \mathcal{P}}(\q)$.
We obtain the advertised map $\psi$
and compose with
the projection $\pi$ from \cref{BlockAutoProjection}:
Define
$$\pi'= \psi\circ \pi: 
\ \ \Aut_{\gr}(S_\q(V)) 
\xrightarrow{\ \ \pi\ \ } 
\Stab(\q)
\xrightarrow{\ \ \psi\ \ } 
\Stab_{\, \mathcal{P}}(\q).
$$
Note that 
$\ker \pi' \subset
\prod_{D\in \mathcal{P}}
\GL(V_D)$
since
$\pi'(h)(D) = \{i: h_{ij}\neq 0 \text{ for } j \in D\}$
by construction.

Using \cref{IsoConditionOnIndices}, as $\q$ is constant on off-diagonal $\mathcal{P}$-minors, one may verify that
$$\ker \pi' = \Aut_{\gr}(S_\q(V)) \cap \prod_{D\in \mathcal{P}}
\GL(V_D) = \prod_{D \in \mathcal{P}}\Aut_{\gr}(S_{\q_{_{DD}}}(V_D))\, .$$

    We define an injective group homomorphism \(\iota':\Stab_{\, \mathcal{P}}(\q) \rightarrow \Aut_{\gr}(S_\q(V))\).
    For each $\sigma \in \Stab_{\, \mathcal{P}}(\q)$, set \(\iota'(\sigma)\) to be the unique invertible linear map on \(V\) which permutes the basis elements of \(V\), sends \(V_D\) to \(V_{\sigma(D)}\), and preserves the order \(v_1 < v_2 < \dots < v_n\) of basis elements within each block of $\mathcal{P}$, using the fact that \(|D| = |\sigma(D)|\).
    We verify that 
    $\iota'(\sigma)$ lies in $\Aut_{\gr}(S_\q(V))$
    using the argument in the proof of \cref{BlockAutoInjection}.

    A straightforward check confirms that
    \(\pi'\iota'\) is the identity,
    and the result follows from the Splitting Lemma, see
\cref{SemiDirectProducts}.
\end{proof}

\vspace{1ex}

\begin{remark}\em{
    For any parameter matrix \(\q\), there are two partitions which satisfy the condition in
    \cref{ConstantOffDiagonalBlocks} trivially, namely \(\p_\q\) itself and \([n]\). Using the partition \(\mathcal{P} = \p_\q\), we recover the result in \cref{BlockMatrixDecomposition}. 
    At the other extreme, using \(\mathcal{P} = [n]\)
    gives the trivial statement \(\Aut_{\gr}(S_\q(V)) \cong \Aut_{\gr}(S_\q(V)) \rtimes \mathfrak{S}_1\). 
    Thus, we seek a partition coarser than \(\p_\q\) and finer than \([n]\) satisfying the hypothesis of 
    \cref{ConstantOffDiagonalBlocks}.
    In \cref{TwoExamplesWithBlocks}, 
    one such example is
    the partition \(\mathcal{P} = \{\{1,2,3\}, \{4,5,6\}, \{7,8,9\}\}\)
    with \(\Aut_{\gr}(S_\q(V))
    \cong
        \Aut_{\gr}(S_{\q'}(V))^3 
        \rtimes
        \langle (1\, 2\, 3) \rangle\)
    for \(\q' = \begin{smallpmatrix}
        1 & 2 & 3 \\
        3 & 1 & 2 \\
        2 & 3 & 1
    \end{smallpmatrix}\).
}
\end{remark}

\vspace{1ex}

\begin{cor}
    Say $\mathcal{P}$ is a partition
    of \([n]\) refining \(\p_\q\)
and
each off-diagonal $\mathcal{P}$-minor of $\q$
is constant 
and shares no entries with 
any diagonal $\mathcal{P}$-minor. 
    Then
    $$\Aut_{\gr}(S_\q(V)) \ \cong\ \prod_{D \in \mathcal{P}}\Aut_{\gr}(S_{\q^{}_{_{DD}}}(V_D))\ \rtimes \ \Stab_{\, \mathcal{P}}(\q)$$
    where \(\Stab_{\, \mathcal{P}}(\q) = \{\sigma \in \mathfrak{S}_{|\mathcal{P}|}:
    |\sigma(D)| = |D| \text{ and } \q^{}_{_{\sigma(D)\sigma(E)}} = \q^{}_{_{DE}} \text{ for all } D,E \in \mathcal{P}\}\).
\end{cor}
\begin{proof}
    The claim follows from
    \cref{ConstantOffDiagonalBlocks} after verifying that 
    $\Stab(\q)$ fixes $\mathcal{P}$.
    Take \(\sigma \in \Stab(\q)\) and
   \(D\in \mathcal{P}\). 
   Since \(D\) is a disjoint union of blocks of \(\q\)
   and \(\sigma\) preserves the sizes of those blocks, \(|\sigma(D)| = |D|\).
   Thus
   \(\q_{_{DD}} = \q_{_{\sigma(D)\sigma(D)}}\)
    as \(\sigma\) stabilizes $\q$.
   Since \(\q_{_{DD}}\) shares no entries with off-diagonal $\mathcal{P}$-minors, it follows that \(\sigma(D)\) is a subset of a single block of $\mathcal{P}$.
   Then \(\sum_{D \in \mathcal{P}}|\sigma(D)| \leq \sum_{D \in \mathcal{P}}|D|\), 
   which implies
   \(\sigma(D)\) is again a block of $\mathcal{P}$
   for each block $D$ of $\mathcal{P}$.
   Thus $\sigma \mathcal{P} = \mathcal{P}$.
\end{proof}

\vspace{1ex}

\begin{remark}
\label{UniponentAutos}
{\em

We write  $\Aut_{\uni}(S_\q(V))$ 
for the set of unipotent automorphisms
of $S_\q(V)$, i.e.,
$\phi$ in $\Aut(S_\q(V))$
such that 
$\phi(v_i)$ is $v_i$ plus terms of
graded degree $>1$
for all $i$,
see
Ceken, Palmieri, Wang, Zhang \cite{Ceken2016}. 
If \(\q\) 
has no rows of all \(1\)'s, then 
by
\cite[Lemma~3.2]{Ceken2016}
$$
\Aut(S_\q(V)) 
\ =\ 
\Aut_{\uni}(S_\q(V))
\,\rtimes\, 
\Aut_{\gr}(S_\q(V)) 
\, ,
$$
and if there are no nontrivial algebraic relations among the parameters \(q_{ij}\), 
then \(\Aut_{\uni}(S_\q(V))\) is trivial and $\Aut(S_\q(V)) = 
\Aut_{\gr}(S_\q(V))$
by \cite[Theorem~3.4]{Ceken2016}. 
See also
Yakimov \cite[Corollary~3.7]{Yakimov2014}. 
} 
\end{remark}

\vspace{1ex}


\section{Infinite Families of 
Graded Automorphism Groups}
\label{InfiniteFamiliesSection}
In this section, we exhibit certain families of 
monomial graded automorphism groups.
We will see
in \cref{EarlierTable} 
that
 nonmonomial graded automorphism groups are 
determined by monomial groups acting in lower dimension,
so these families are helpful for classification results.
These families have the form 
$\{\k^n \rtimes G_n\}$
for $G_n$ a subgroup of $\mathfrak{S}_n$,
with each group arising as 
$\Aut_{\gr}(S_\q(\k^n))$
for some \(n \times n\) 
quantum parameter matrix \(\q\).

\begin{prop}\label{SymmetricGroupFamily}
For $n\geq 1$ and 
 \(\textup{char}(\k) \neq 2\), there is a unique $n \times n$ quantum parameter matrix \(\q\) 
with
    $\Stab(\q)=\mathfrak{S}_n$. This yields the isomorphism
    $$
    \Aut_{\gr}(S_\q(\k^n)) 
   \ \cong\ (\k^\times)^n\, \rtimes\, \mathfrak{S}_n
    \, . 
    $$
    In fact, \(\mathfrak{S}_n\) is the only \(2\)-transitive permutation group 
    acting on $[n]$ which is \( \Stab(\q)\)
    for some quantum parameter matrix \(\q\).
\end{prop}
\begin{proof}
    By the 2-transitivity of $\mathfrak{S}_n$, $\Stab(\q) = \mathfrak{S}_n$ if and only if 
    the rows of \(\q\) are distinct 
    and \(q_{ij} = q_{\ell m}\) 
    for all \(i \neq j\) and \(\ell \neq m\). 
    Thus, $\Stab(\q) = \mathfrak{S}_n$ exactly when \(q_{ij} = -1\) for all \(i \neq j\) since \(q_{ji}=q_{ij} = q_{ji}^{-1}\) implies  \(q_{ij} = -1\) for all \(i \neq j\) when 
    the rows are distinct, and the isomorphism
    follows  from
\cref{BlockMatrixDecomposition}.
    For the second claim, 
    we use \cref{BlockAutosSubsetMaximal} 
    after noting that    
    \([n]^2/G = [n]^2/\mathfrak{S}_n\)
    for any 2-transitive group $G$.
    \end{proof}

\vspace{1ex}

\begin{remark}{\em 
    For \(n > 3\),
    \cref{SymmetricGroupFamily} shows that no quantum parameter matrix \(\q\) exists with
    $\Aut_{\gr}(S_\q(V))  \cong \prod_{B\in \p_\q}\GL(V_B)
    \,\rtimes\, {\rm Alt}_n
    $
    for $V=\k^n$ and
    $\Stab(\q) =  \rm{Alt}_n$, 
    the alternating group on \(n\) elements,
  as  \({\rm Alt}_n\) is 2-transitive on $[n]$.
    }
\end{remark}

\vspace{1ex}

We now consider the dihedral group \(D_{2n}\) 
    of order \(2n\).
\begin{prop}\label{DihedralFamily}
For \(\textup{char}(\k) \neq 2\) 
and \(n \geq 4\), 
there exists a quantum parameter matrix 
\(\q\) such that 
\(\Aut_{\gr}(S_\q(\k^n)) \cong (\k^\times)^n \rtimes D_{2n} \).
    \end{prop}
    \begin{proof}
        For \(n = 4\), let
        $ \q $  
        $=\begin{scriptsize}       
    \begin{smallpmatrix}
            1 & 1 & -1 & 1 \\
            1 & 1 & 1 & -1 \\
            -1 & 1 & 1 & 1 \\
            1 & -1 & 1 & 1
        \end{smallpmatrix}
        \end{scriptsize}$.
        Since all the rows of \(\q\) are distinct, by \cref{BlockMatrixDecomposition},
        \[
        \Aut_{\gr}(S_\q(\k^4)) = (\k^\times)^4\, \rtimes\, 
        \{\sigma \in \mathfrak{S}_4:\ q_{ij} = q_{\sigma(i)\sigma(j)}
        \  \text{ for } 1 \leq i,j \leq 4\}
        \, .
        \]
        A short calculation shows that 
        the group on the right is generated by
        $(1\, 2\, 3\, 4)$ and \((1\, 3)\) 
        and thus
        \[\Aut_{\gr}(S_\q(\k^4)) 
        \ \cong\
        (\k^\times)^4\, \rtimes\, 
        \langle (1\, 2\, 3\, 4), (1\, 3) \rangle \ \cong\
        (\k^\times)^4\, \rtimes\, D_{8}\, .
        \]
        
        For \(n > 4\), let \(\q\) be the \(n \times n\) quantum parameter matrix in which every entry is \(1\) except for the superdiagonal and subdiagonal entries, and the top-right and bottom-left entries, which are \(-1\):
        \[\q = \begin{smallpmatrix}
            \ 1 & -1 & \ 1 & \hdots & \ 1 & -1 \\
            -1 & \ 1 & -1 & \hdots & \ 1 & \ 1 \\
           \ 1 & -1 & \ 1 & \ddots & \ 1 & \ 1 \\
            \vdots & \vdots & \ddots & \ddots & \vdots & \vdots \\
            \ 1 & \ 1 &\ 1 & \hdots & \ 1 & -1 \\
             -1 & \ 1 & \ 1 & \hdots & -1 & \ 1
        \end{smallpmatrix}\, .\]
        As the rows of \(\q\) are distinct, 
        \(\Aut_{\gr}(S_\q(\k^n)) \cong (\k^\times)^n \rtimes \Stab (\q)
        \)
        by \cref{BlockMatrixDecomposition},
        where $\Stab (\q)$
         is $\{\sigma \in \mathfrak{S}_n: q_{ij} 
        = q_{\sigma(i)\sigma(j)}\, \text{for } 1 \leq i,j \leq n\}\).
        In the following calculations, we take 
        indices mod n. 
        The group
        \(D_{2n}\) is generated by
        \((1\, 2\, \, \cdots\, \, n)\) 
        and 
        \(\tau := (1\ n)(2\ n-1) \dots \), 
        the product of 2-cycles defined by \(\tau(i) = 1-i \text{ mod } n\). 
        These lie in \(\Stab (\q)\)
        as \(\q\) is constant on 
        super diagonals and is symmetric. 
        
   Conversely, suppose
        $\sigma$ lies in $\Stab (\q)$.
    Then
        $-1=q_{i (i+1)}=q_{\sigma(i)\sigma(i+1)}$
        so our choice of $\q$ implies that
        $\sigma(i+1)-\sigma(i)
        \equiv 
        \pm 1
        \text{ mod } n$
        for $1\leq i<n$.
        But 
        $\sigma(i+1)-\sigma(i)\equiv 1$
        implies
        $\sigma(i+2)-\sigma(i+1)\equiv 1$
        as well,
        otherwise 
        $\sigma(i+2)\equiv 
        \sigma(i+1)-1
        \equiv \sigma(i)+1-1\equiv \sigma(i)$
        which is impossible.
        Likewise
$\sigma(i+1)-\sigma(i)\equiv -1$
        implies
        $\sigma(i+2)-\sigma(i)\equiv -1$
        as well.
     Then as
        $\sigma(i)$ is the telescoping sum
        $\sigma(1)+
        (\sigma(2)-\sigma(1)) +
        \ldots + 
        (\sigma(i)-\sigma(i-1))$,
 either
        $\sigma(i) \equiv 
        \sigma(1)+(i-1) \text{ mod } n$
        for all $1\leq i \leq n$
        or else
        $\sigma(i) \equiv 
        \sigma(1)-(i-1)
        \text{ mod } n$
                for all $1\leq i \leq n$.
        In the first case,
 $\sigma$
        is a power of
        $(1\, 2\, \cdots\, n)$,
        and in the second,
        a power of
        $(1\, 2\, \cdots\, n)$
                multiplied by
        $\tau$.
        In either case, $\sigma$ lies in $D_{2n}$.
    \end{proof}

Lastly, we turn to cyclic groups.

\begin{prop}\label{CyclicAutoGroups}
 For $n\geq 1$, let $G$ be a cyclic subgroup of \(\mathfrak{S}_n\).
    For \(|\k|\) sufficiently large, with \(\textup{char}(\k) \neq 2\) whenever $|G|$ is even, there is a quantum parameter matrix \(\q\) with \(\Aut_{\gr}(S_\q(\k^n)) \cong (\k^\times)^n \rtimes G\).
\end{prop}
\begin{proof}
Suppose $\sigma$ in $\mathfrak{S}_n$ generates $G$.
First note that 
    \(\sigma \, \q = \q\) 
    if and only if 
    \( (\tau \sigma \tau^{-1}) (\tau\, \q) = \tau\, \q\)
    for all $\tau$ in $\mathfrak{S}_n$,
    with \(\tau\,\q\) another quantum parameter matrix. Thus, without loss of generality, we may assume \(\sigma\) has a disjoint cycle decomposition
    \[\sigma 
    = 
    (1\ 2\ \cdots \ m_1)(m_1 + 1 \ \cdots \ m_2)
    \ \cdots\ 
    (m_{r-1} + 1\ \cdots\ n)\, \]
    for some $m_i$.
    For \(1 \leq i,j \leq n\), set $c_i$ to be the length of the cycle
    containing $i$ and write
    $i\sim j$ for $i$ and $j$ in same
    orbit, i.e., in same cycle, so \(\sigma(i) \equiv i + 1 \textup{ mod } c_i\).
    For \(1 \leq i \leq j \leq n\) with \(i \sim j\),
    select $q_{ij}\neq 0$ 
    satisfying
    the rules $q_{ij} = 1$ if $i=j$, 
    $q_{ij} = -1$ if $j-i = c_i/2$, 
    and, for all \( \ell \leq m \) with \(i \sim \ell \sim m\), $q_{ij} = q_{\ell m}$
    exactly when 
    $j-i \equiv m-\ell \text{ mod } c_i$.
    After these have been chosen, for \(i \not \sim j\), select $q_{ij}\neq 0$ satisfying the rule that, 
    for all \(\ell \leq m\), 
    $q_{ij}=q_{\ell m}$ 
    exactly when
    $i \sim \ell$, $j \sim m$,
    and $j-i\equiv m-\ell \text{ mod } 
    \gcd(c_i, c_j)$.
    Set $q_{ji}=q_{ij}^{-1}$ and
    assume further that the $q_{ij}$ are chosen
    so that $q_{ij} \neq q_{m\ell}$
    for all $i\leq j$ with $i\not \sim j$
    and all $\ell \leq m$
    to guarantee that $\Stab(\q)$ is not too large.
    Observe that \(\q = (\q_{ij})\) is 
    then a quantum parameter matrix.

    We now argue that \(\Aut_{\gr}(S_\q(\k^n)) \cong (\k^\times)^n \rtimes G\).
    All rows of \(\q\) are distinct, since \(1\) only appears as an entry of \(\q\) on the diagonal. Thus
    \(\Aut_{\gr}(S_\q(\k^n)) \cong (\k^\times)^n \rtimes \Stab (\q)\) by \cref{BlockMatrixDecomposition}.
    We argue that $G=\Stab (\q)$.
    For fixed \(i < j\), \(i \sim \sigma(i)\) and \(j \sim \sigma(j)\). 
    Furthermore, \(\sigma(i) - i \equiv 1 \textup{ mod } c_i\) and \(\sigma(j) - j \equiv 1 \textup{ mod } c_j\) and so \(\sigma(j) - \sigma(i) \equiv j - i \textup{ mod } \gcd(c_i, c_j)\).
    Thus, \(q_{ij} = q_{\sigma(i)\sigma(j)}\)
    and \(\langle \sigma \rangle \subset \Stab (\q)\).
    Conversely, suppose \(\tau \in \Stab (\q)\). 
    Then \(q_{\tau(i)\tau(j)} = q_{ij}\)
    and
    \(
    \tau(j) - \tau(i) \equiv j - i \
    \textup{ mod } \gcd(c_i, c_j)
       \)
    so 
    $$ \tau(j) - j \ \equiv\ \tau(i) - i 
    \mod
    \gcd(c_i, c_j)
     \qquad\text{ for all } i,j.
    $$
    By the Chinese Remainder Theorem for non co-prime moduli, see~\cite{OreGCRT}, 
    there exists a positive integer \(N\) 
    (unique modulo \(\lcm(c_1, \hdots, c_n)=|G|\))
    with
    \(\tau(i) - i \equiv N \textup{ mod } c_i\) for all \(i\).
    Hence \(\tau = \sigma^N \in G\).
    Thus, \(\Stab (\q) = G\) and so \(\Aut_{\gr}(S_\q(\k^n)) \cong (\k^\times)^n \rtimes G\).
\end{proof}

\vspace{1ex}

\begin{example}{\em 
For $G$ the cyclic group generated
by \((1\,2\,3\,4\,5\,6)(7\,8)\),
 \(|\k| \geq 15\), and \(\textup{char}(\k) \neq 2\), the last proof gives
\(\Aut_{\gr}(S_\q(\k^9)) = (\k^\times)^9
\rtimes G\) for \(a,b,c,d,e,f\) and their inverses distinct in \(\k \) and
\begin{tiny}
\[
\q =
\left(
\begin{array}{cccccc|cc|c}
  1 & a & b & -1 & b^{-1} & a^{-1} & c & d & e\\
  a^{-1} & 1& a & b & -1 & b^{-1} & d & c & e\\
  b^{-1} & a^{-1} & 1 & a & b & -1 & c & d & e\\
  -1 & b^{-1} & a^{-1} & 1 & a & b & d & c & e\\
   b & -1 & b^{-1} & a^{-1} & 1 & a & c & d & e\\
 a & b  & -1 & b^{-1} & a^{-1} & 1 & d & c & e\\
   \hline\rule{0ex}{2.5ex}
 c^{-1} & d^{-1} &  c^{-1} & d^{-1} &  c^{-1} & d^{-1} & 1 & -1 & f \\
 d^{-1} & c^{-1} & d^{-1} &  c^{-1} & d^{-1} &  c^{-1}  & - 1 & 1 & f \\
   \hline\rule{0ex}{2.5ex}
  e^{-1} & e^{-1} & e^{-1} &  e^{-1} & e^{-1} &  e^{-1}  & f^{-1} & f^{-1} & 1 \\ 
\end{array}
\right)
\]
\end{tiny}

    }
\end{example}

\vspace{1ex}


\section{Direct Product Decomposition}
\label{DirectProductSection}

We now decompose 
the graded automorphism
group $\Aut_{\gr}(S_\q(V))$
of a quantum affine space
into graded automorphism groups 
of lower dimensional
spaces
using results in \cref{SemiDirectSection},
and we verify that the set of 
graded automorphism groups
is closed under direct products.
We again fix \(V \cong \k^n\) and an \(n \times n\) quantum parameter matrix \(\q\). 
Various partitions of the set \([n]\) indexing the rows of \(\q\) 
describe automorphisms.
For example, the partition \(\p_\q\) on \([n]\) identifies the ``elementary automorphisms" of \cite{KKZ}, 
and the block-circle decomposition of \([n]\) 
in \cite{KKZ} 
identifies
the important class of mystic reflection subgroups of \(\Aut_{\gr}(S_\q(V))\). 
We construct here a partition $\mathcal{P}$
of \([n]\) giving a decomposition of \(\Aut_{\gr}(S_\q(V))\) 
into a direct product of 
groups 
\(\Aut_{\gr}(S_{\q_{_{DD}}}
(V))\)
ranging over the blocks $D$ of $\mathcal{P}$.

The group \(\Aut_{\gr}(S_\q(V))\) 
is completely determined by the blocks of $\q$
and the stablizing permutation group $\Stab(\q)$.
However, computing $\Stab(\q)$
concretely may be computationally prohibitive.
Indeed, $\Stab(\q)$
can be found with an algorithm of time complexity 
$O(n!)$, and we seek an alternative
description for \(\Aut_{\gr}(S_\q(V))\)
that bypasses that calculation.
The partition $\mathcal{P}$ introduced
here
only requires
an algorithm of time complexity $O(n^2)$, see \cref{PartitionConstruction}.

Observe that
every element $\sigma\in \Stab(\q)$
moves a row of \(\q\) 
to another row which is a permutation of the first. Moreover, if these two rows differ at a column, then \(\sigma\) can not
fix that column.
This motivates the next definition.

\vspace{1ex}

\begin{definition}\label{TransitiveClosure}
{\em
 Define an equivalence relation \(\mathcal{R}\) 
on \([n]\) by taking the symmetric closure of the 
relation with \(i\sim \ell \) 
whenever there is a chain \(i = i_0, i_1 , \dots ,i_p = \ell \)
for some $p\geq 0$ 
with, for all \(1 \leq s \leq p\),
 \begin{enumerate}
     \item the row of \(\q\) indexed by 
     \(i_{s-1}\) is a permutation of the row indexed by \(i_{s}\),
     or
     \item 
     $q_{i_{s-1}i_{s} } 
     \neq q_{i_t i_{s}} \) for some
     $0\leq t < s$.
      \end{enumerate}
    }
\end{definition}

\vspace{1ex}

\begin{thm}\label{DirectProductDecomp}
Let $\q$ be a quantum parameter matrix.
The graded automorphism group breaks into
a direct product over blocks
of the partition \(\mathcal{P}\) 
of $[n]$ given by 
the equivalence classes of
\(\mathcal{R}\):
    \[\Aut_{\gr}(S_\q(V)) \ \cong\
    \prod_{D \in\, \mathcal{P}} \Aut_{\gr}
    \big(S_{\q_{_{DD}}}
    (V_D)\big)
    \, .
    \]
\end{thm}
\begin{proof}
    We verify that the partition \(\mathcal{P}\)  satisfies the hypothesis of \cref{ConstantOffDiagonalBlocks} and that the group \(G = \{\sigma \in \mathfrak{S}_{|\mathcal{P}|}: |\sigma(D)| = D \textup{ and } \q_{_{\sigma(D)\sigma(E)}} = \q_{_{DE}} \textup{ for all blocks } D,E \textup{ of } \mathcal{P}\}\) given there is trivial.

    First notice that \(\mathcal{P}\) is refined by \(\p_\q\) since if two rows of $\q$ are equal, then their
    indices are 
    equivalent under $\mathcal{R}$ and lie in the same block of \(\mathcal{P}\).
    Next we observe that 
    \(\q\) is constant on the off-diagonal 
    $\mathcal{P}$-minors.
    To see this, take \(D,E \in \mathcal{P}\) with \(D \neq E\) 
    and let \(i,\ell \in D\) and \(j,m \in E\). 
    Then 
    $q_{im} = q_{\ell m}$
    since otherwise a chain 
    $i=i_0, i_1, \ldots, \ell$ or $\ell=\ell_0,\ell_1,\ldots, i$
    as in \cref{TransitiveClosure}
    could be extended to a chain also satisfying the
    condition of \cref{TransitiveClosure}
    by appending $m$ to the end,
    which would
    force $i$ and $m$ in the same equivalence
    class of $\mathcal{R}$ and $D=E$.
    Similarly, \(q_{im} = q_{ij}\),
    and thus
    \(q_{ij} = q_{im} = q_{\ell m}\). 
    So \(\q\) is constant on the 
off-diagonal $\mathcal{P}$-minor \(\q_{_{DE}}\) of $\q$.

    Let \(\sigma \in \Stab(\q)\). Define \(\tau \in \mathfrak{S}_n\) to be the unique permutation 
    preserving the order of indices within each block 
    with \(\tau(B) = \sigma(B)\) for all blocks \(B\) of $\q$.
    If the row indexed by \(i\) is a permutation of the row indexed by \(j\), 
    say by \(\tau'\), 
    then \(\q_{\tau(i)\tau(m)} = \q_{im} 
    = \q_{j\tau'(m)}
    = \q_{\tau(j)\tau\tau'(m)}\), 
    and so the rows indexed by \(\tau(i)\) and \(\tau(j)\) are equal up to a permutation as well. 
    Similarly, if \(\q_{im} \neq \q_{jm}\), then \(\q_{\tau(i)\tau(m)} \neq \q_{\tau(j)\tau(m)}\). 
    Thus, whenever a chain \(i = i_0, i_1 , \dots , i_m = j\) exists as in \cref{TransitiveClosure},  \(\tau(i) = \tau(i_0), \tau(i_1), \dots \tau(i_m) = j\) is also a chain
    satisfying the definition.
    Thus $\sigma(D) \in \mathcal{P}$  for all \(D \in \mathcal{P}\).

   To see that
   $G$ is trivial, take \(\sigma \in G\) and a block \(D\) of \(\mathcal{P}\). Then \(\q_{_{\sigma(D)\sigma(E)}} = \q_{_{DE}}\) for all blocks \(E\) of \(\mathcal{P}\), which implies that each row indexed by an element of
   \(D\) is a permutation of a row indexed by 
   an element of \(\sigma(D)\).
   Thus the elements of $D$ are equivalent to those of $\sigma(D)$
   under the relation $\mathcal{R}$, 
and \(D = \sigma(D)\) by the definition of \(\mathcal{P}\).
\end{proof}

\vspace{1ex}

\begin{remark}\label{PartitionConstruction}\em{
Let $\mathcal{P}$ denote the partition
giving the equivalence classes of $\mathcal{R}$ (\cref{TransitiveClosure}). From the proof of \cref{DirectProductDecomp}, we see that $\mathcal{P}$ satisfies three key properties:
\begin{enumerate}
    \item $\mathcal{P}$ is refined by $\p_\q$;
    \item indices of rows of \(\q\) that are permutations of each other lie in the same block of $\mathcal{P}$;
    \item for any two blocks $D,E$ of $\mathcal{P}$ with $D \neq E$, the $\mathcal{P}$-minor $\q_{_{DE}}$ is constant.
\end{enumerate}
Observe that $\mathcal{P}$
is the finest partition 
with these properties.
Indeed, we construct \(\mathcal{P}\) 
with the following two step process.
First, we construct the partition of \([n]\) 
whose blocks 
index rows of \(\q\) which are permutations of each other.
We do this by constructing the corresponding multiset of entries for each row, hashing these multisets, and then grouping indices by hash value.
This has $O(n^2)$ time complexity for multiset construction and $O(n)$ time complexity for the grouping.

For the second step, we successively merge blocks $D,E$  whenever \(q_{_{DE}}\) is not a constant matrix.
The algorithm terminates with a partition $\mathcal{P}$ such that the $\mathcal{P}$-minor $\q_{_{DE}}$ is constant for all distinct blocks $D,E$ of $\mathcal{P}$. This step has $O(n^2)$ time complexity since the number of comparisons is bounded by the number of entries of $\q$. This results in an algorithm generating $\mathcal{P}$ with $O(n^2)$ time complexity overall.
}
\end{remark}

\vspace{1ex}

\begin{example}{\em 

Suppose $a,b,c,d,f$ and their inverses are distinct in $\k$
and
$$\q = \begin{smallpmatrix}
    1 & -1 & a & b & c & c & f & f\\
    -1 &1 & b & a & c & c & f & f\\
    a^{-1} & b^{-1} & 1 & -1 & c & c & d & f \\
    b^{-1} & a^{-1} & -1 & 1 & c & c & f & d \\
    c^{-1} & c^{-1} & c^{-1} & c^{-1} & 1 & -1 & c^{-1} & c^{-1} \\
    c^{-1} & c^{-1} & c^{-1} & c^{-1} & -1 & 1 & c^{-1} & c^{-1} \\
    f^{-1} & f^{-1} & d^{-1} & f^{-1} & c & c & 1 & -1 \\
    f^{-1} & f^{-1} & f^{-1} & d^{-1} & c & c & -1 & 1
\end{smallpmatrix}
\, .
$$
Note that any \(\sigma\in \Stab(\q)\) can only only permute 
blocks whose rows
are permutations of each other,
so it must fix the sets 
$\{1,2\}$,  $\{3,4\}$,  $\{5,6\}$, and $\{7,8\}$.
We turn to the off-diagonal minors. 
If $\sigma$ interchanges $1$ and $2$,
 then the block $\begin{smallpmatrix}
    a & b \\ b & a
\end{smallpmatrix}$ is preserved 
which implies that $\sigma$
interchanges $3$ and $4$.
 But then the block $\begin{smallpmatrix}
    d & f \\ f & d
\end{smallpmatrix}$ is preserved 
so $\sigma$ must interchange $7$ and $8$.
So $\Stab(\q) 
\subset \langle (1\, 2)(3\, 4)(7\, 8), (5\, 6)\rangle$. The corresponding partition 
$\mathcal{P}$ of \cref{DirectProductDecomp}
has blocks 
$D=\{1,2,3,4,7,8\}$ and $E=\{5,6\}$,
and
$\Aut_{\gr}(S_\q(V))
\cong
\Aut_{\gr}(S_{\q_{_{DD}}}(V_D))
\times
\Aut_{\gr}(S_{\q_{_{EE}}}(V_E))
$.
}
\end{example}

\vspace{1ex}

 \cref{DirectProductDecomp}
implies that the set of
graded automorphism groups 
is closed under direct products.

\begin{cor}
\label{ClosedUnderDirectProducts}
    Let \(V_1\) and \(V_2\) be vector spaces 
    with \(\dim V_1 = n_1\) and \(\dim V_2 = n_2\). 
    Let \(\q_1\) and \(\q_2\) be \(n_1 \times n_1\) and \(n_2 \times n_2\) quantum parameter matrices,
    respectively.
    For \(|\k|\) sufficiently large, there exists a quantum parameter matrix \(\q\) of size \((n_1 + n_2) \times (n_1 + n_2)\) such that 
    \[
    \Aut_{\gr}(S_\q(V_1 \oplus V_2)) 
    \ \cong \  
    \Aut_{\gr}(S_{\q_1}(V_1)) \times \Aut_{\gr}(S_{\q_2}(V_2))
    \, .\]
\end{cor}
\begin{proof}
We take a basis for \(V=V_1 \oplus V_2\) by appending a basis for \(V_2\) 
onto a basis for \(V_1\) after embedding $V_i$ in $V$.
        Select $a$ not in $\{0, -1, 1\}$ 
        nor an entry of both \(\q_1\) and \(\q_2\). 
        We set \(\q = \begin{smallpmatrix}
         \q_1 & A \\
         A' & \q_2
     \end{smallpmatrix}\) 
     where \(A\) is the \(n_1 \times n_2\) constant matrix with all entries \(a\) 
     and \(A'\) 
     is the \(n_2 \times n_1\) constant matrix with all entries \(a^{-1}\). 
     Then \(\q\) is a quantum parameter matrix corresponding to \(V_1 \oplus V_2\).

     Write \(\mathcal{P}_1, \mathcal{P}_2\), and \(\mathcal{P}\) for the partitions of \([n_1], [n_2]\), and \([n_1 + n_2]\) corresponding to \(\q_1, \q_2\), and \(\q\), respectively, 
     arising from the decomposition given by \cref{DirectProductDecomp}.
     
      Fix indices \(i,j,m \in [n_1 + n_2]\).
      Since $A$ and $A'$ are constant matrices, rows $i$ and 
      $j$ of $\q$ are permutations of each other if and only if 
they correspond to permuted rows both in $\q_1$ or both in $\q_2$.
      Therefore, \(\mathcal{P} = \mathcal{P}_1 \sqcup \mathcal{P}_2\). 
         
         Since
     \(\q_{_{DD}} = {(\q_i)}_{_{DD}}\) for \(D \in \mathcal{P}_i\),
     \[
     \prod_{D \in \mathcal{P}}\Aut_{\gr}(S_{\q_{_{DD}}}((V_1 \oplus V_2)_{D})) 
     \ \cong\ 
     \prod_{D \in \mathcal{P}_1}\Aut_{\gr}(S_{{(\q_1)}_{_{DD}}}(V_D))\ \times\ \prod_{D \in \mathcal{P}_2}\Aut_{\gr}(S_{{(\q_2)}_{_{DD}}}(V_D))
     \]
     and thus
          $\Aut_{\gr}(S_\q(V_1 \oplus V_2)) \cong \Aut_{\gr}(S_{\q_1}(V_1)) \times \Aut_{\gr}(S_{\q_2}(V_2))$
          by \cref{DirectProductDecomp}.
\end{proof}

\vspace{1ex}

\begin{remark}\em{
    The requirement that $|\k|$ be sufficiently large in the last proposition is only to guarantee some $a \in \k^\times$ is not an entry of both $\q_1$ and $\q_2$ with $a \neq a^{-1}$.
    \cref{sec: MinimalFieldSize} offers a more detailed discussion on how the cardinality of $\k$ determines which graded automorphism groups appear for some quantum parameter matrix $\q$ with entries in $\k$.
}
\end{remark}

\vspace{1ex}


\section{Kronecker products
of quantum parameters}
\label{TensorProductSection}

We investigate here the graded automorphism
group arising from the  Kronecker product of quantum parameter matrices.
This gives a useful tool for constructing 
quantum affine spaces with desirable automorphism groups.

We fix two quantum parameter matrices
    \(\q \in \Mat(n, \k)\) and 
    \(\q' \in \Mat(n', \k)\) 
    for integers $n, n'\geq 1$
    and index each entry of  $\q \ot \q' $ 
    in $\Mat(n n', \k)$
    by a pair of indices $(i,i')$ and $ (j,j')$,
    setting
    \begin{equation}
    \label{PairsOfOrderedPairs}
    (\q \ot \q')_{(i,i'),(j,j')} 
    = q_{ij}\, q'_{i'j'}\, .
    \end{equation}
    Note that  \(\q \ot \q'\) is again a
    quantum parameter matrix
    and
    defines
        the skew polynomial algebra 
    \(S_{\q \ot \q'}(V \ot V')\) 
    with relations 
    \[
    (v_j \ot w_m)(v_i \ot w_\ell) 
    = 
    q_{ij}\, q'_{\ell m}\
    (v_i \ot w_\ell)(v_j \ot w_m)\quad \textup{ for } 1 \leq i,j\leq n \textup{ and } 1 \leq \ell,m \leq n'
    \, 
    \]
for basis 
$v_1,\ldots, v_{n}$ of $V$ and
basis
$w_1,\ldots, w_{n'}$ of $V'$,
using the basis $v_i\ot w_\ell$ of
$V\ot V'$.
Generically, a pair of quantum parameter matrices $\q$ and $\q'$
are independent 
in the sense that they share no entries
nor even products of pairs
    of entries
    except $1$, i.e.,
    $$
    \{ ab: a, b \text{ entries of $\q$}  \}
    \ \cap\  
        \{ cd: c, d \text{ entries of $\q'$}\}  
= \{1\}
    \, .
    $$
    We formalize
    this notion in the next definition.
        
 \vspace{1ex}

\begin{definition}{\em
 We say that \(\q\) and \(\q'\) are 
 \textit{multiplicatively independent} 
 when $\q \ot \q$ and $\q' \ot \q'$ have $1$ as their only common entry.
    }
\end{definition}

\vspace{1ex}

We use the indexing convention of \cref{PairsOfOrderedPairs}
to describe blocks of $\q\ot \q'$:

\begin{lemma}\label{MultIndDecompSplit}
    The partition \(\p_{\q \ot \q'}\) refines the product partition \(\{B\times B': B\in \p_\q, B' \in \p_{\q'}\}\) of \([n] \times [n']\). If \(\q\) and \(\q'\) are multiplicatively independent, then
    these partitions 
    of $[n]\times [n']$ are equal.
   \end{lemma}
\begin{proof}
    Suppose \(i\) and \(\ell\) are in the same block of $\q$ and \(i'\) and \(\ell'\) are in the same block of $\q'$.
    Then \(q_{im}\, q'_{i'm'} = q_{\ell m}\, q'_{\ell'm'}\) and thus
    \((\q \ot \q')_{(i,i')(m,m')}
    =  
    (\q \ot \q')_{(\ell,\ell')(m,m')}\) for all \((m,m')
    \in [n] \times [n']\),
    so that \((i,i')\) and \((\ell,\ell')\) are in the same block of 
    $\q \ot \q'$. 
    Refinement in the opposite direction uses multiplicative independence: 
        For all $(m,m')$, if
\(q_{im}\, q'_{i'm'} 
    = q_{\ell m}\, q'_{\ell' m'}\),  
   then  \(q_{im}\, q_{m\ell} = q_{\ell'm'}'\, q_{m'i'}'\), 
    and hence 
    \(q_{im}\, 
    q_{m\ell} = 1 = q_{\ell'm'}'\, q_{m'i'}'\)
 so
    \(q_{im} = q_{\ell m}\)  and \(q'_{i'm'} = q'_{\ell'm'}\).
\end{proof}

\vspace{1ex}


We now describe the
graded automorphism group arising from the tensor product of two quantum parameter matrices with multiplicatively independent entries.
\begin{thm}\label{TensorProductAutos}
    For \(\q\) and \(\q'\) 
    multiplicatively independent,
    $\Stab(\q\ot\q')
    \cong \Stab (\q) \times \Stab (\q')$
    and
    \[
    \Aut_{\gr}(S_{\q \ot \q'}(V \otimes V')) 
    \ \cong\  
    \bigg(\prod_{B \in \p_\q,\, B' \in \p_{\q'}
    \rule{0ex}{2ex}}
    \GL(V_B \otimes V'_{B'})\bigg)
    \,  \rtimes\, 
    (\Stab (\q) \times \Stab (\q')).
    \]
\end{thm}
\begin{proof}
     By \cref{MultIndDecompSplit}, 
    \(\p_{\q \ot \q'}
    =
    \{B\times B': B\in \p_\q, B' \in \p_{\q'}\}\).
     The subspace of \(V \ot V'\) corresponding to the block \(B \times B'\) is $(V\ot V')_{B\times B'} = V_B\otimes V_{B'}$
     with our indexing convention.
     \cref{BlockMatrixDecomposition} then gives
     \[
     \Aut_{\gr}(S_{\q \ot \q'}(V \otimes V')) 
     \ \cong\ 
     \big(\prod_{B \in \p_\q,\, B' \in \p_{\q'}
    \rule{0ex}{2ex}}
     \GL(V_B \otimes V'_{B'})\big)
     \, \rtimes\, \Stab (\q \ot \q')\ .\]
     Recall
     $
     \Stab (\q \ot \q') 
     = 
     \{\tau \in \mathfrak{S}_{rr'}:
     \tau
     \text{ preserves block sizes
     for $\q\ot \q'$}
     \text{ and }
     \tau  (\q \ot \q') = \q \ot \q' \}$
     for $r=|\p_\q|$ and $r'=|\p_{\q'}|$.
     Here
   preserving
     block sizes means
      $|\tau(B,B')|=|B||B'|$
     for all blocks $B$ of $\q$
     and $B'$ of $\q'$,
     in which case
      $\tau(\q \ot \q')_{_{(B,B')(C,C')}}$ is  $(\q \ot \q')_{_{\tau(B,B')\, \tau(C,C')}}$.
      We identify \(\mathfrak{S}_r \times \mathfrak{S}_{r'}\)  with a subgroup of \(\mathfrak{S}_{rr'}\),
   and, for $\sigma \in \mathfrak{S}_r$
   and $\sigma'\in \mathfrak{S}_{r'}$
  preserving blocks sizes of $\q$ and $\q'$, respectively,
  we define the quantum parameter matrix
     \[
     (\sigma, \sigma')(\q \ot \q')\  =\  \sigma\q\ot \sigma'\q'
     \, . 
     \]  
   We argue that 
   $\Stab (\q) \times \Stab (\q')
     = \Stab (\q \ot \q') $
     under this identification.
     
     If \((\sigma, \sigma') \in \Stab (\q) \times \Stab (\q')\), then
     $
     (\sigma, \sigma') (\q \ot \q') = \sigma \q \ot \sigma' \q' = \q \ot \q'$
     so that $(\sigma, \sigma') \in \Stab (\q \ot \q')$.
    
     For the opposite inclusion, let \(\tau \in \Stab (\q \ot \q')\).
     Define \(\pi_1, \pi_2\) to be the projection functions onto the first and second coordinates, respectively.
 We claim that for any fixed $B'$ in $\p_{q'}$, 
 the function
 $\sigma: \p_{\q}\rightarrow \p_{\q}$
 defined by $B\mapsto \pi_1\, \tau(B,B')\)
 is independent of choice of $B'$.

 To verify the claim, we fix
 $B', C'$ in $\p_{\q'}$
 and $B$ in $\p_\q$, and show
 $\pi_1\tau(B,B')= \pi_1\tau(B, C')$.
 Write
     \(\tau(B,B') = (D,D')\)
     and
     \(\tau(B,C') = (F, F')\)
     for some blocks $D, D', F, F'$.
     Take any $(M,M')$ in $\p_q \times
     \p_q'$ and set
 $(L,L') = \tau^{-1}(M, M')$.
     Then 
     \[
     (\q \ot \q')_{_{(B,B')\, (L,L')}} 
     = (\q \ot \q')_{_{(D,D')\, (M,M')}} 
     \quad\text{ and }\quad 
    (\q \ot \q')_{_{(B,C')\, (L,L')}} 
    = 
     (\q \ot \q')_{_{(F,F')\, (M,M')}} 
 \, . \]
Since each minor matrix considered
 has all identical entries,
 we take a sample entry from each.
 Say 
 $b, b', c', d, d', f, f',
 \ell, \ell', m, m'$
 lie in 
 $B, B', C', D, D', F, F',
 L, L', M, M'$,
 respectively.
 The first equality in the display
 gives
 $q_{b\ell} \, q'_{b'\ell'}
 =q_{dm} \, q'_{d'm'}$
 which forces
 $q_{b\ell}=q_{dm}$
 by the multiplicative independence
 of $\q$ and $\q'$.
 A similar argument using the second
 equality 
 confirms that
 $q_{b\ell}=q_{fm}$,
 and hence
 $q_{dm}=q_{fm}$.
 As we may take $m$ to be arbitrary,
 $D=F$.
 Thus
 $\sigma = \pi_1\tau(-, B'):  \p_\q\rightarrow \p_\q$
 does not depend on choice of $B'$
 in $\p_{\q'}$.

 Likewise,  we may define
 a function
 $\sigma': \p_{\q'}\rightarrow \p_{\q'}$
 defined by $B'\mapsto \pi_2\, \tau(B,B')$
 for a fixed block $B$ of $\p_\q$
and show that $\sigma'$ does not depend on 
 choice of $B$.
 
 Then $\tau: \p_{\q\ot \q'}
 \rightarrow \p_{\q\ot {\q'}}$
 may be written as
 $\tau=(\sigma, \sigma')$.
 As $\tau$ is bijective,
both $\sigma$ and $\sigma'$ must be bijective.
 For example,
 if $\sigma(B_1)=\sigma(B_2)$,
 then for any block $C'$,
 $\tau(B_1, C')
 =(\sigma(B_1), \sigma'(C'))
 =(\sigma(B_2), \sigma'(C'))
 = \tau(B_2, C')$
 and $B_1=B_2$ as $\tau$ is one-to-one.
 Hence $\sigma$ and $\sigma'$ 
 are permutations of the blocks of $\q$
 and $\q'$, respectively.

We next argue that $\sigma$ lies in $\Stab(\q)$
and $\sigma'$ lies in $\Stab(\q')$.
First observe that both $\sigma$ and $\sigma'$
 preserve block size.
 Indeed,
 for all blocks $B$ in $\p_\q$ and 
 $B'$ in $\p_{\q'}$,
 $|B||B'|=|\tau(B,B')|
 =|\sigma(B)| |\sigma'(B')|$
 as $\tau\in \Stab(\q\ot \q')$.
 Thus if either $\sigma$ or $\sigma'$
 preserve block size, then so does the other.
 And if neither $\sigma$ nor $\sigma'$
 preserve block size, then 
 there are blocks $B$ and $B'$ sent by $\sigma$
 and $\sigma'$, respectively, 
 to larger blocks,
  implying that $|B||B'|=|\tau(B,B')|
=|\sigma(B)| |\sigma'(B')|> |B| |B'|$.
Since $\tau\in \Stab(\q\ot \q')$,
for any blocks $B, L$ of $\q$ and $B', L'$ of $\q'$,
$$
(\q \ot \q')_{_{(B,B')\, (L,L')}} 
=
(\q \ot \q')_{_{\tau(B,B')\, \tau(L,L')}} 
=(
\q \ot \q')_{_{(\sigma(B), \sigma'(B'))\, (\sigma(L), \sigma'(L'))
}}
$$
which implies that
$\q_{_{B L}} = \q_{_{\sigma(B)\, \sigma(L)}}$
and 
$\q'_{_{B'\, L'}} 
= \q'_{_{\sigma(B')\, \sigma(L')}}$
since $\q$ and $\q'$ are 
multiplicatively independent and $\sigma$ and $\sigma'$
preserve block sizes.
Thus $\tau=(\sigma,\sigma')$ lies in 
$\Stab (\q) \times \Stab (\q')$
and hence
   $\Stab (\q \ot \q') 
     \subset \Stab (\q) \times \Stab (\q')$.
\end{proof}

\vspace{1ex}

\begin{remark}{\em
\label{InfiniteFamilyHugeField}
For $|\k|$ sufficiently large
and
any  $n\times n$ quantum parameter matrix $\q$,
\cref{TensorProductAutos} implies that
there is an $m\times m$ quantum
parameter matrix $\q'$
for all positive multiples $m$ of $n$
with
stabilizing permutation group
$\Stab(\q') \cong \Stab(
\q)$.
Indeed,
we may find $\q''$ of appropriate size
with $\q$ and $\q''$ multiplicatively independent and $\Stab(\q'')$ trivial by \cref{CyclicAutoGroups}.
}
\end{remark}

\vspace{1ex}

The last theorem gives infinite families of graded automorphism groups that all share the same stabilizing permutation group.
\begin{cor}
\label{BlockPermutationTypeAppearsInfinitelyOften}
For positive integers \(n,m\), and \(n_1, \dots, n_r\), 
and a permutation group $G$,
there exists a quantum parameter matrix
\(\q \in \Mat(n, \k)\) 
with $\Stab(\q)=G$,
giving
$$
\Aut_{\gr}(S_\q(\k^n)) 
\ \cong\ 
\prod_i \GL(n_i, \k)\, \rtimes\, G
\, ,
$$
if and only if 
there exists a quantum parameter matrix
\(\q' \in \Mat(mn, \k)\) 
with $\Stab(\q')=G$,
giving
$$ 
\Aut_{\gr}(S_{\q'}(\k^{mn})) \ \cong \ 
\prod_i \GL(m\cdot n_i, \k)\, \rtimes\, G
\, .
$$
\end{cor}
\begin{proof}
    Suppose \(\q\) is an $n\times n$
    quantum parameter matrix 
   with
    \(G = \Stab (\q)\) 
    giving the semidirect
    product indicated by
    \cref{BlockAutosSubsetMaximal}.
Let \(\q' = \q \ot \mathds{1}_m\) 
    for $\mathds{1}_m$ the $m \times m$ matrix whose every entry is $1$. Since $\mathds{1}_m$ and \(\q\) are multiplicatively independent, and \(\Stab (\mathds{1}_m)\) is trivial, \cref{TensorProductAutos} 
    implies that \(\Stab (\q) \cong \Stab(\q')\). By \cref{MultIndDecompSplit}, since \(\p_{\mathds{1}_m}\) has only one block, 
    \(\p_{\q'}\) 
    is in bijection with \(\p_\q\), 
    and each block of \(\p_{\q'}\) has \(m\) times as many elements as the corresponding block of \(\p_\q\). 
    Thus, 
    \[\Aut_{\gr}(S_{\q'}(\k^{nm})) \ \cong \ 
\prod_i \GL(m\cdot n_i, \k)\, \rtimes\, G
\, .\]

    Conversely, fix a quantum 
    parameter matrix \(\q' \in \Mat(mn, \k)\)
    with the above graded automorphism group.
    Then every block of \(\p_{\q'}\) has size a multiple of \(m\). As all entries in \(\q'_{_{BC}}\) are equal for all blocks \(B,C \in \p_{\q'}\), 
     there is a \(n \times n\) quantum parameter matrix \(\q\)
     with \(\q' = \q \ot \mathds{1}_m\)
     and a similar argument verifies that \(\Aut_{\gr}(S_\q(\k^n))\) is as claimed.
\end{proof}

\vspace{1ex}

\begin{example}{\em 
Observe $\Aut_{\gr}(S_\q(\k^3)) \cong (\k^\times)^3 \rtimes \mathfrak{S}_3\, $ 
   and
   $\Aut_{\gr}(S_{\q'}(\k^9)) \cong (\GL(3, \k))^3 \rtimes \mathfrak{S}_3$
for
$$\q =         \begin{smallpmatrix}
        1 & -1 & -1 \\
        -1 & 1 & -1 \\
        -1 & -1 & 1
    \end{smallpmatrix}
\qquad\text{ and }\qquad
\q'= \q \ \ot 
            \begin{smallpmatrix}
        1 & 1 & 1 \\
        1 & 1 & 1 \\
        1 & 1 & 1
    \end{smallpmatrix}.
    $$
    }
   \end{example}

\vspace{1ex}

\section{Classification of Graded Automorphism Groups of Low Dimension}
\label{ClassificationSection}
We classify graded automorphism groups of 
quantum affine spaces
up to $\dim V\leq 7$.
Since two skew polynomial rings
are isomorphic as algebras
if and only if their
quantum parameters differ
by a permutation applied to indices, see Gaddis \cite[Theorem~7.4]{gaddis2014isomorphisms},
we classify these groups up to a permutation of the basis elements of \(V\).
Determining the classification reduces to computing the possibilities for the stabilizing permutation group 
$\Stab(\q)$ of an $n \times n$ quantum parameter matrix $\q$ up to conjugation in $\mathfrak{S}_n$.
Not every subgroup of the symmetric group can appear as a stabilizing permutation group. For example, of the $96$ conjugacy classes of
subgroups of $\mathfrak{S}_7$, only 
$53$ appear.
This may be explained by the
following corollary
of
\cref{BlockAutosSubsetMaximal}.

\begin{cor} 
\label{BlockAutosOrbitProperty}
Consider a skew polynomial algebra
 \(S_\q(V)\) with $r$ the number of 
 distinct rows of $\q$.
 The group of graded automorphism is
    $
    \Aut_{\gr}(S_\q(V)) 
    \cong 
    \big(\prod_{i=1}^r
    \GL(n_i, \k)\big)\rtimes G
    $ for
    a group \(G \subset \mathfrak{S}_r\) 
       with  the property
    that $H \subset G$
whenever $[r]^2/H = [r]^2/G$
 for all
$    H\subset \mathfrak{S}_r$.
Here, $G=\Stab(\q)$.
\end{cor}

We used 
the software GAP \cite{GAP} 
to give conjugacy
classes of subgroups of $\mathfrak{S}_n$
and a depth-first search algorithm in Python \cite{Python} to compute the graded automorphism groups of all possible $\q$ matrices, using \cref{BlockAutosOrbitProperty}
to rule out groups
when helpful. The core linear algebra operations were performed with NumPy \cite{Numpy}.

\vspace{1ex}

\subsection*{Maximality condition
is not sufficient}

Not every group $G$ with the
maximality property
of
\cref{BlockAutosOrbitProperty}
is the stabilizing permutation group
of some $\Aut_{\gr}(S_\q(V))$,
as we see in the next example.

\vspace{1ex}

\begin{example}{\em 
    Suppose
     \(G = \langle (1\, 2)(3\, 4), (1\, 3)(2\, 4) \rangle \subset \mathfrak{S}_r\)
     for $r\geq 4$.
    Then \(G\) satisfies the maximality condition
     of \cref{BlockAutosOrbitProperty}.
     We argue $G \neq \Stab(\q)$ for every
     quantum parameter matrix $\q$.
     Fix $\q$
     with blocks $B_1, \ldots, B_r$
     and $G = \Stab(\q)$. 
     Define a new $4 \times 4$ matrix $\q'$ with $q'_{ij}$ the entry in $\q_{_{B_iB_j}}$ for $i, j\leq 4$. 
    Note that
     $\q_{_{B_iB_m}} = \q_{_{B_jB_m}}$ for $i,j \leq 4$ and $m >4$
     as $G$ is transitive on $\{1,2,3,4\}$ but fixes $m$.
     This implies that
     $\q'$ is a quantum parameter matrix with distinct rows
      and $\Stab(\q')=G$.
 Thus
$$\q' = 
    \begin{smallpmatrix}
        1 & a & b & c \\
        a & 1 & c & b \\
        b & c & 1 & a \\
        c & b & a & 1
    \end{smallpmatrix}
\qquad\text{ 
    for some }
    a,b,c \in \{1, -1\}
    \, .
    $$
     Note \(c \neq b\) 
    else $(1\, 2)$ would lie in $G$. 
    If $(b,c)=(1,-1)$,
    then
 $a\neq 1$, 
    else $(2\, 3)$ would lie in $G$, and 
 $a\neq -1$,
 else rows $1$ and $3$ would be identical.
Similarly impossible is the case $(b,c)=(-1,1)$.
    Hence there is no quantum parameter matrix \(\q\)
    with $G= \Stab(\q)$ and
    \(\Aut_{\gr}(S_\q(V)) \cong \prod_{i=1}^r\GL(n_i, \k) \rtimes G\).
    }
\end{example}

\vspace{1ex}

\subsection*{Tracking minimal field size}\label{sec: MinimalFieldSize}
Whether or not a group
may arise as $\Aut_{\gr}(S_\q(\k^n))\) for some quantum parameter matrix \(\q\) depends on the cardinality of the field \(\k\).
See \cref{CyclicAutoGroups} for example:
The field must be large enough 
for \(\q\) to contain enough distinct entries. For example, if \(\k = \mathbb{F}_2\), then
the only
quantum parameter matrix is 
the trivial matrix
\(\q = \mathds{1}_n\)
with
\(S_\q(V) \cong \k[v_1, \dots,  v_n]\) and \(\Aut_{\gr}(S_\q(\k^n)) \cong \GL(n,\k)\). A sharp lower bound on \(|\k|\) is calculated for each group. For $n > 7$, if $G$ indeed does appear as $\Stab(\q)$ for some $\q$,
the number of orbits of $G$ on $[n]^2$ is an upper bound on the minimum
cardinality of $\k^\times$ required, 
but we are unaware of any formulas for
these bounds.

\vspace{1ex}

\begin{example}\label{SillyMissingGroup}{\em
    Let \(n=4\) and \(H = (\k^\times)^4
    \rtimes\langle (1\,2\,3) \rangle\). 
    We argue that $4$ is the minimal cardinality of $\k$ 
    to guarantee
    $H$ arises as the graded automorphism
    group of some skew polynomial algebra. 
        We may rule out
    \(|\k| = 2\), 
    and
    there are four possibilities for $\q$
    when
    \(|\k| = 3\),
    namely
    \[\begin{smallpmatrix}
        1 & 1 & 1 & 1 \\
        1 & 1 & 1 & 1 \\
        1 & 1 & 1 & 1 \\
        1 & 1 & 1 & 1
    \end{smallpmatrix},
    \quad \begin{smallpmatrix}
        1 & 1 & 1 & -1 \\
        1 & 1 & 1 & -1 \\
        1 & 1 & 1 & -1 \\
        -1 & -1 & -1 & 1
    \end{smallpmatrix},
    \quad \begin{smallpmatrix}
        1 & -1 & -1 & 1 \\
        -1 & 1 & -1 & 1 \\
        -1 & -1 & 1 & 1 \\
        1 & 1 & 1 & 1
    \end{smallpmatrix},
    \quad 
    \text{ and }\quad
    \begin{smallpmatrix}
        1 & -1 & -1 & -1 \\
        -1 & 1 & -1 & -1 \\
        -1 & -1 & 1 & -1 \\
        -1 & -1 & -1 & 1
    \end{smallpmatrix}
    \]
    which give rise to respective
    graded automorphism groups
    \[\GL(4,\k),
    \quad \GL(3,\k)\times \k^\times,
    \quad (\k^\times)^4\, \rtimes\, \langle (1\,2), (1\,2\,3) \rangle,
    \quad\text{ and }\quad
    (\k^\times)^4\, \rtimes\, \mathfrak{S}_4\, .\]
    Of these, none are 
    isomorphic to \(H\),
    so the minimal cardinality of $\k$
    is at least $4$. 
     A simple computation verifies that \(\Aut_{\gr}(S_\q(V)) \cong H\)
     for
    \(\k = \{0,1,a,a^{-1}\} \cong \mathbb{F}_4\) and
    $ \q = \begin{smallpmatrix}
        1 & a & a^{-1} & a \\
        a^{-1} & 1 & a & a \\
        a & a^{-1} & 1 & a \\
        a^{-1} & a^{-1} & a^{-1} & 1
    \end{smallpmatrix}\, .$
    }
\end{example}

\vspace{1ex}

\subsection*{Nonmonomial groups
classified using monomial groups
of lower dimension}
By \cref{BlockAutosSubsetMaximal}, we
may assign to each quantum parameter matrix
\(\q\) a partition
$\lambda$
of $n=\dim V$
with
\(\Aut_{\gr}(S_\q(V)) \cong \big(\prod_i \GL(\lambda_i, \k)^{m_i}\big) \rtimes G\)
for $G=\Stab(\q)$, the stabilizing permutation group,
and $\lambda =(\lambda_1^{m_1} \ \ldots\ \lambda_s^{m_s})$.
Here, 
$\lambda_1, \ldots, \lambda_s$ are the distinct 
parts of $\lambda$ with $m_i$ the multiplicity
of $\lambda_i$ so that the length of $\lambda$ 
is $r=\sum_{i=1}^s m_i$.

Every {\em monomial} graded automorphism group 
has the form \((\k^\times)^n \rtimes G\)
corresponding to the partition 
\(\lambda = (1\ \ldots\ 1)\). 
For {\em nonmonomial} graded automorphism groups
corresponding to a fixed $\lambda$,
the stabilizing permutation groups \(G\) 
all arise from
{monomial}
graded automorphism groups
in lower dimension $r = \ell(\lambda)$,
as explained with the next proposition.

\begin{prop}\label{EarlierTable}
Suppose $|\k|$ is sufficiently large
and let 
$\lambda=(\lambda_1^{m_1} \ \ldots\
\lambda_s^{m_s})$ 
be a partition of $n$ of length $r$.
Let $G$ be a subgroup of $\mathfrak{S}_r$.
The following are equivalent:
\begin{itemize}
    \item 
There is a
$n\times n$ quantum parameter matrix 
$\q$ with block sizes corresponding to $\lambda$
with $\Stab(\q)=G$.
\item
There is an $r\times r$
quantum parameter matrix
$\q'$ with \(\Stab(\q') = G\) 
a subgroup of $\prod_{i=1}^s \mathfrak{S}_{m_i}$
in $\mathfrak{S}_r$
(up to conjugation)
with $\Aut_{\gr}(S_{\q'}(V))$ monomial.
\end{itemize}
\end{prop}
\begin{proof}
 We use \cref{BlockAutosSubsetMaximal}.
Take $\q$ as described 
and label its blocks  $B_1, \ldots, B_r$ so that 
each $\q_{_{B_i B_j}}$
has all entries the same.
Let $\q'$ be the $r\times r$ quantum parameter matrix
whose $ij$-th entry is the entry
in $\q_{_{B_i B_j}}$.
In other words, $\q'$ is 
the matrix obtained by removing duplicate rows and columns of $\q$ and reindexing.
Let \(\mathcal{P}\) be the partition of \([r]\) 
with two indices \(i,j\) in the same block 
exactly when \(|B_i| = |B_j|\). 
We write \(\mathcal{P} = \{E_1, E_2, \dots E_s\}\) and replace entries in $\q'$
so that entries that were equal remain equal and entries that were distinct remain distinct in 
each off-diagonal minor $\q'_{E_i E_j}$
and so that 
each of these
minors 
has distinct entries from the rest of \(\q'\).
Then \(\sigma \in \mathfrak{S}_r\)
lies in $\Stab(\q')$
exactly when \(\sigma(i)\) and \(i\) are in the same block of \(\mathcal{P}\) and
$\q'_{ij}=\q'_{\sigma(i) \sigma(j)}$
for all $i,j$. This happens
 exactly when \(|B_i| = |\sigma(B_i)|\) and
\(\q_{_{B_i B_j}} = \q_{_{\sigma(B_i)\sigma(B_j)}}\) 
for all $i,j$. Therefore \(\Stab(\q') = \Stab(\q) = G\).
Finally, since \(\q'\) has distinct rows, the graded automorphism group corresponding to \(\q'\) is monomial.

Now assume $\q'$ as described in the statement is given.
Then $\q'$ must have distinct
rows as $\Aut_{\gr}(S_{\q'}(V))$
is monomial.
We construct $\q$ from $\q'$ by
repeating $\lambda_i$ times 
the $i$-th row of $\q'$
to create a matrix of size $n\times r$
and then adding the columns necessary
to obtain a
$n\times n$ quantum parameter matrix $\q$
as claimed
with $\Stab(\q)=\Stab(\q')$.
\end{proof}

\vspace{1ex}

\begin{example}{\em 
    Say $\q$
is an $8\times 8$ quantum parameter
matrix corresponding
to the partition
\(\lambda = (1,1,2,2,2) = (1^2\ 2^3) \)
so $\ell(\lambda)=5$.  We
find in the $\dim V = 5$ table
the  rows with  $\Stab \q'$ a subgroup
of $\mathfrak{S}_2\times \mathfrak{S}_3$
identified with a subgroup of 
$\mathfrak{S}_5$
up to conjugation in $\mathfrak{S}_5$,
as
there must be two blocks of the smallest size and
three of the next smallest size.
Thus, up to reindexing,
$\Stab(\q)=G$ and
$$
\Aut_{\gr}(S_\q(\k^8)) 
\ \cong\  
\big(
(\k^\times)^2 \times \GL(2, \k)^3\big) \,\rtimes\, G\, 
$$
for $G\subset \mathfrak{S}_2
\times \mathfrak{S}_3$
one of $8$ groups, namely,
$1$,
$1\times\langle (1\, 2\,  3)\rangle$,
$1\times \mathfrak{S}_3$,
$\mathfrak{S}_2\times 1$,
$\mathfrak{S}_2\times\langle (1\, 2)\rangle$,
$\mathfrak{S}_2\times\langle (1\, 2\,  3)\rangle$,
$\langle (1\, 2) \times (1\, 2) \rangle$,
and $\mathfrak{S}_2\times \mathfrak{S}_3$
for $1$ the identity group.
All of these $8$ groups arise for some $\q$.
}
\end{example}

\vspace{1ex}

\subsection*{Listing the monomial
groups, counting the nonmonomial groups}
For $\dim V \leq 5$, we give
each graded automorphism group
\(\Aut_{\gr}(S_\q(V)) \)
explicitly.
For $\dim V = 6$ and $\dim V=7$,
we give the monomial 
graded automorphism groups  
$(\k^{\times})^m \rtimes G$
by listing the
stabilizing permutation groups $G=\Stab(\q)$
that arise explicitly.
We count the
nonmonomial graded automorphism groups
in the classification using \cref{EarlierTable},
which 
explains how to write them out
explicitly 
using earlier tables
in order to complete the classification.
We omit explicit tables 
for these nonmonomial groups 
for brevity.

\vspace{1ex}

\subsection*{The explicit classification}
Below, we classify the matrix groups that arise as graded automorphism groups for skew polynomial algebras with $\dim V \leq 7$.
We give the classification up to a permutation of the basis elements, i.e., up to a permutation of indices 
of the quantum parameter matrices.
In the tables, 
\((\k^\times)^m \times \GL(\ell, \k)\)  indicates a group of block diagonal matrices, which sometimes appears in a semidirect
product with a permutation group acting by permuting the blocks.
There are exactly 

%
\begin{itemize}
    \item 
\(3\) groups for $\dim V =2$,
\item
 \(6\) groups for $\dim V = 3$,
 \item
 $15$ groups for $\dim V = 4$,
 \item
 $25$ groups for $\dim V = 5$,
 \item 
 $65$ groups for $\dim V = 6$, and
 \item
 $105$ groups for $\dim V = 7$
 \end{itemize}
that arise as 
\(\Aut_{\gr}(S_\q(V))\) 
for some quantum parameter matrix $\q$.
For $\dim V = 3$, this recovers
\cite[Theorem~11.1]{LevandovskyyShepler}.
Aside from a handful of permutation groups of small order, the classification  for $\dim V$ equal to $4$, $5$, $6$, and $7$ was determined by exhaustive computation, taking only a brief amount of time with current 
computer capabilities.

In each table below, 
we list 
the groups in the classification
for a fixed dimension of $V$
and
indicate a sample  $\q$
giving that group as 
\(\Aut_{\gr}(S_\q(V))\) 
along with the
stabilizing permutation group
$G=\Stab(\q)$ as in 
\cref{BlockAutosSubsetMaximal}.
We also indicate
the minimum cardinality of \(\k\) 
required.
For each matrix $\q$ listed, we assume $a,b,c,d,e,f$
are pairwise distinct scalars in $\k^*/ \{\pm 1\}$.

We also indicate in each table the orbits 
of $G$ on \([n]^2\) for $n=\dim V$
using the fact that $G$ permutes blocks of $\q$ of the same size:
For each possible block size,
we reindex so the blocks of that size are 
$B_1,\ldots, B_m$
and give an $m\times m$ matrix
indicating the orbits of $G$ on $[m]\times [m]$;
the entries 
range from $1$ to the 
number of orbits
with $ij$-th and $k \ell $-th entries equal 
whenever $(i,j)$ and $(k, \ell)$ lie in the same orbit.

\begin{footnotesize}
\begingroup
        \renewcommand\arraystretch{2.7}
        \centering
        \begin{longtable}{|ccccc|}
        \caption{Graded Automorphism Groups for $\dim V=2$} \\
        \hline
        \(\Aut_{\gr}(S_\q(V))\) & \(\q\) & 
        $\Stab \q$
         & Orbits & Min \(|\k|\) \\
        \hline
        \((\k^{\times})^2\) & \(\begin{smallpmatrix}
            1 & a \\
            a^{-1} & 1
        \end{smallpmatrix}\) & \(\langle e \rangle\) & \(\begin{bsmallmatrix}
            1 & 2 \\
            3 & 4
        \end{bsmallmatrix}\) & 4 \\


        \((\k^{\times})^2 \rtimes \langle (1\,2)\rangle\) & \(\begin{smallpmatrix}
            1 & -1 \\
            -1 & 1
        \end{smallpmatrix}\) & \(\langle (1\,2) \rangle\) & \(\begin{bsmallmatrix}
            1 & 2 \\
            2 & 1
        \end{bsmallmatrix}\) & 3 \\


        \(\GL(2, \k)\) & \(\begin{smallpmatrix}
            1 & 1 \\
            1 & 1
        \end{smallpmatrix}\) & \(\langle e \rangle\) & \(\begin{bsmallmatrix}
            1
        \end{bsmallmatrix}\) & 2
        
        \\
        \hline
        \end{longtable}
    \endgroup
\end{footnotesize}

\begin{footnotesize}
\begingroup
        \renewcommand\arraystretch{2.7}
        \centering
        \begin{longtable}{|ccccc|}
        \caption{Graded Automorphism Groups for $\dim V=3$} \\
        \hline
        \(\Aut_{\gr}(S_\q(V))\) & \(\q\) & \(\Stab \q \) & Orbits & Min \(|\k|\)\\
        \hline
        \((\k^{\times})^3\) & \(\begin{smallpmatrix}
            1 & a & b \\
            a^{-1} & 1 & c \\
            b^{-1} & c^{-1} & 1
        \end{smallpmatrix}\) & \(\langle e \rangle\) & \(\begin{bsmallmatrix}
            1 & 2 & 3 \\
            4 & 5 & 6 \\
            7 & 8 & 9
        \end{bsmallmatrix}\) & 4 \\
        
        
        \((\k^{\times})^3 \rtimes \langle (1\,2) \rangle\) & \(\begin{smallpmatrix}
            1 & -1 & a \\
            -1 & 1 & a \\
            a^{-1} & a^{-1} & 1
        \end{smallpmatrix}\) & \(\langle( 1\, 2)\rangle\) & \(\begin{bsmallmatrix}
            1 & 2 & 3 \\
            2 & 1 & 3 \\
            4 & 4 & 1
        \end{bsmallmatrix}\) & 3 \\

        
        \((\k^{\times})^3 \rtimes \langle (1\,2\,3) \rangle\) & \(\begin{smallpmatrix}
            1 & a & a^{-1} \\
            a^{-1} & 1 & a \\
            a & a^{-1} & 1
        \end{smallpmatrix}\) & \(\langle(1\,2\,3)\rangle\) & \(\begin{bsmallmatrix}
            1 & 2 & 3 \\
            3 & 1 & 2 \\
            2 & 3 & 1
        \end{bsmallmatrix}\) & 4 \\
        
        
        \((\k^{\times})^3 \rtimes \mathfrak{S}_3\) & \(\begin{smallpmatrix}
            1 & -1 & -1 \\
            -1 & 1 & -1 \\
            -1 & -1 & 1
        \end{smallpmatrix}\) & \(\mathfrak{S}_3\) & \(\begin{bsmallmatrix}
            1 & 2 & 2 \\
            2 & 1 & 2 \\
            2 & 2 & 1
        \end{bsmallmatrix}\) & 3 \\
        

        \(\k^{\times} \times \GL(2,\k)\) & \(\begin{smallpmatrix}
            1 & a & a  \\
            a^{-1} & 1 & 1 \\
            a^{-1} & 1 & 1
        \end{smallpmatrix}\) & 
        \(\langle e \rangle\) & \(\begin{bsmallmatrix}
            1
        \end{bsmallmatrix}\), \(\begin{bsmallmatrix}
            1
        \end{bsmallmatrix}\) & 3 \\
        
        
        \(\GL(3,\k)\) & \(\begin{smallpmatrix}
            1 & 1 & 1 \\
            1 & 1 & 1 \\
            1 & 1 & 1
        \end{smallpmatrix}\) & \(\langle e \rangle\) & \(\begin{bsmallmatrix}
            1
        \end{bsmallmatrix}\) & 2
        
        \\
        \hline
        \end{longtable}
    \endgroup
\end{footnotesize}

\newpage

For \(\dim V = 4\),
the classification comprises $15$ graded automorphism groups.

\begin{footnotesize}
    \begingroup
        \renewcommand\arraystretch{2.7}
        \centering
        \begin{longtable}{|ccccc|}
        \caption{Graded Automorphism Groups for $\dim V = 4$} \\
        \hline
        \(\Aut_{\gr}(S_\q(V))\) & \(\q\) & \(\Stab \q\) & Orbits & Min \(|\k|\)\\
        \hline
        \((\k^{\times})^4\) & \(\begin{smallpmatrix}
            1 & a & b & c \\
            a^{-1} & 1 & d & e \\
            b^{-1} & d^{-1} & 1 & f \\
            c^{-1} & e^{-1} & f^{-1} & 1
        \end{smallpmatrix}\) & \(\langle e \rangle\) & \(\begin{bsmallmatrix}
            1 & 2 & 3 & 4 \\
            5 & 6 & 7 & 8 \\
            9 & 10 & 11 & 12 \\
            13 & 14 & 15 & 16
        \end{bsmallmatrix}\) & 5 \\
        
        
        \((\k^{\times})^4 \rtimes \langle (1\,2) \rangle\) & \(\begin{smallpmatrix}
            1 & -1 & b & c \\
            -1 & 1 & b & c \\
            b^{-1} & b^{-1} & 1 & d \\
            c^{-1} & c^{-1} & d^{-1} & 1
        \end{smallpmatrix}\) & \(\langle (1\,2) \rangle\) & \(\begin{bsmallmatrix}
            1 & 2 & 3 & 4 \\
            2 & 1 & 3 & 4 \\
            5 & 5 & 6 & 7 \\
            8 & 8 & 9 & 10
        \end{bsmallmatrix}\) & 3 \\
        \((\k^{\times})^4 \rtimes \langle (1\,2\,3) \rangle\) & \(\begin{smallpmatrix}
            1 & a & a^{-1} & b \\
            a^{-1} & 1 & a & b \\
            a & a^{-1} & 1 & b \\
            b^{-1} & b^{-1} & b^{-1} & 1
        \end{smallpmatrix}\) & \(\langle (1\,2\,3) \rangle\) & \(\begin{bsmallmatrix}
            1 & 2 & 3 & 4 \\
            3 & 1 & 2 & 4 \\
            2 & 3 & 1 & 4 \\
            5 & 5 & 5 & 6
        \end{bsmallmatrix}\) & 4 \\
        
        
        \((\k^{\times})^4 \rtimes \langle (1\,2\,3), (1\,2) \rangle\) & \(\begin{smallpmatrix}
            1 & -1 & -1 & b \\
            -1 & 1 & -1 & b \\
            -1 & -1 & 1 & b \\
            b^{-1} & b^{-1} & b^{-1} & 1
        \end{smallpmatrix}\) & \(\langle (1\,2\,3), (1\,2) \rangle\) & \(\begin{bsmallmatrix}
            1 & 2 & 2 & 3 \\
            2 & 1 & 2 & 3 \\
            2 & 2 & 1 & 3 \\
            4 & 4 & 4 & 5
        \end{bsmallmatrix}\) & 3 \\
        \((\k^{\times})^4 \rtimes \langle (1\,2)(3\,4)\rangle\) & \(\begin{smallpmatrix}
            1 & \pm1 & a & b \\
            \pm1 & 1 & b & a \\
            a^{-1} & b^{-1} & 1 & \pm1 \\
            b^{-1} & a^{-1} & \pm1 & 1
        \end{smallpmatrix}\) & \(\langle (1\,2)(3\,4) \rangle\) & \(\begin{bsmallmatrix}
            1 & 2 & 3 & 4 \\
            2 & 1 & 4 & 3 \\
            5 & 6 & 7 & 8 \\
            6 & 5 & 8 & 7
        \end{bsmallmatrix}\) & 3 \\
        
        
        \((\k^{\times})^4 \rtimes \langle (1\,2),(3\,4)\rangle\) & \(\begin{smallpmatrix}
            1 & -1 & a & a \\
            -1 & 1 & a & a \\
            a^{-1} & a^{-1} & 1 & -1 \\
            a^{-1} & a^{-1} & -1 & 1
        \end{smallpmatrix}\) & \(\langle (1\,2),(3\,4) \rangle\) & \(\begin{bsmallmatrix}
            1 & 2 & 3 & 3 \\
            2 & 1 & 3 & 3 \\
            4 & 4 & 5 & 6 \\
            4 & 4 & 6 & 5
        \end{bsmallmatrix}\) & 5 \\
        
        
        \((\k^{\times})^4 \rtimes \langle (1\,2\,3\,4)\rangle\) & \(\begin{smallpmatrix}
            1 & a & \pm1 & a^{-1} \\
            a^{-1} & 1 & a & \pm1 \\
            \pm1 & a^{-1} & 1 & a \\
            a & \pm1 & a^{-1} & 1
        \end{smallpmatrix}\) & \(\langle (1\,2\,3\,4) \rangle\) & \(\begin{bsmallmatrix}
            1 & 2 & 3 & 4 \\
            4 & 1 & 2 & 3 \\
            3 & 4 & 1 & 2 \\
            2 & 3 & 4 & 1
        \end{bsmallmatrix}\) & 5\\
        
        
        \((\k^{\times})^4 \rtimes \langle (1\,2\,3\,4), (1\,3)\rangle\) & \(\begin{smallpmatrix}
            1 & 1 & -1 & 1 \\
            1 & 1 & 1 & -1 \\
            -1 & 1 & 1 & 1 \\
            1 & -1 & 1 & 1
        \end{smallpmatrix}\) & \(\langle (1\,2\,3\,4), (1\,3) \rangle\) & \(\begin{bsmallmatrix}
            1 & 2 & 3 & 2 \\
            2 & 1 & 2 & 3 \\
            3 & 2 & 1 & 2 \\
            2 & 3 & 2 & 1
        \end{bsmallmatrix}\) & 3\\
        
        
        \((\k^{\times})^4 \rtimes \mathfrak{S}_4\) & \(\begin{smallpmatrix}
            1 & -1 & -1 & -1 \\
            -1 & 1 & -1 & -1 \\
            -1 & -1 & 1 & -1 \\
            -1 & -1 & -1 & 1
        \end{smallpmatrix}\) & \( \mathfrak{S}_4\) & \(\begin{bsmallmatrix}
            1 & 2 & 2 & 2 \\
            2 & 1 & 2 & 2 \\
            2 & 2 & 1 & 2 \\
            2 & 2 & 2 & 1
        \end{bsmallmatrix}\) & 3\\
        
        
        \((\k^\times)^2 \times \GL(2, \k)\) & \(\begin{smallpmatrix}
            1 & a & b & b \\
            a^{-1} & 1 & b & b \\
            b^{-1} & b^{-1} & 1 & 1 \\
            b^{-1} & b^{-1} & 1 & 1
        \end{smallpmatrix}\) & \(\langle e \rangle\) & \(\begin{bsmallmatrix}
            1 & 2  \\
            3 & 4  \\
        \end{bsmallmatrix}, \begin{bsmallmatrix}
            1
        \end{bsmallmatrix}\) & 4\\
        
        
        \(((\k^\times)^2 \times \GL(2, \k)) \rtimes \langle (1\,2)\rangle\) & \(\begin{smallpmatrix}
            1 & -1 & a & a \\
            -1 & 1 & a & a \\
            a^{-1} & a^{-1} & 1 & 1 \\
            a^{-1} & a^{-1} & 1 & 1
        \end{smallpmatrix}\) & \(\langle (1\,2) \rangle\) & \(\begin{bsmallmatrix}
            1 & 2  \\
            2 & 1  \\
        \end{bsmallmatrix}, \begin{bsmallmatrix}
            1
        \end{bsmallmatrix}\) & 3\\


        \((\GL(2, \k))^2\) & \(\begin{smallpmatrix}
            1 & 1 & a & a \\
            1 & 1 & a & a \\
            a^{-1} & a^{-1} & 1 & 1 \\
            a^{-1} & a^{-1} & 1 & 1
        \end{smallpmatrix}\) & \(\langle e \rangle\) & \(\begin{bsmallmatrix}
            1 & 2  \\
            3 & 4  \\
        \end{bsmallmatrix}\) & 4 \\
        
        
        \(\left(\GL(2, \k)\right)^2 \rtimes (1\,2)\) & \(\begin{smallpmatrix}
            1 & 1 & -1 & -1 \\
            1 & 1 & -1 & -1 \\
            -1 & -1 & 1 & 1 \\
            -1 & -1 & 1 & 1
        \end{smallpmatrix}\) & \(\langle (1\,2) \rangle\) & \(\begin{bsmallmatrix}
            1 & 2  \\
            2 & 1  \\
        \end{bsmallmatrix}\) & 3 \\
        
        
        \(\k^\times \times \GL(3, \k)\) & \(\begin{smallpmatrix}
            1 & a & a & a \\
            a^{-1} & 1 & 1 & 1 \\
            a^{-1} & 1 & 1 & 1 \\
            a^{-1} & 1 & 1 & 1
        \end{smallpmatrix}\) & \(\langle e \rangle\) & \(\begin{bsmallmatrix}
            1
        \end{bsmallmatrix}, \begin{bsmallmatrix}
                1
            \end{bsmallmatrix}\) & 3\\


        \(\GL(4, \k)\) & \(\begin{smallpmatrix}
            1 & 1 & 1 & 1 \\
            1 & 1 & 1 & 1 \\
            1 & 1 & 1 & 1 \\
            1 & 1 & 1 & 1
        \end{smallpmatrix}\)\rule[-15pt]{0pt}{0pt} & \(\langle e \rangle\) & \(\begin{bsmallmatrix}
            1
        \end{bsmallmatrix}\) & 2 
        \\
        \hline
        \end{longtable}
    \endgroup
\end{footnotesize}

For $\dim V= 5$,
we list monomial 
graded automorphism groups in the first table and nonmonomial groups in the second table. There are total of 25 graded automorphism groups.

   \newpage

   \begin{footnotesize}
    \begingroup
        \renewcommand\arraystretch{2.9}
        \centering
        \setlength{\tabcolsep}{1pt}
        \begin{longtable}{|ccccc|}
            \caption{Monomial 
            Graded Automorphism Groups for $\dim V = 5$} \\
            \hline
                        \(\Aut_{\gr}(S_\q(V))\) & \(\q\) & \(\Stab \q\) & Orbits & Min \(|\k|\) \hspace{3pt} \\
            \hline
            \rule{0ex}{6.5ex}
            \((\k^{\times})^5\) & \ \(\begin{smallpmatrix}
                1 & a & b & c & d\\
                a^{-1} & 1 & e & f & g\\
                b^{-1} & e^{-1} & 1 & h & i \\
                c^{-1} & f^{-1} & h^{-1} & 1 & j \\
                d^{-1} & g^{-1} & i^{-1} & j^{-1} & 1
            \end{smallpmatrix}\) \ \ 
            & \(\langle e \rangle\) & \(\begin{bsmallmatrix}
                1 & 2 & 3 & 4 & 5\\
                6 & 7 & 8 & 9 & 10\\
                11 & 12 & 13 & 14 & 15 \\
                16 & 17 & 18 & 19 & 20 \\
                21 & 22 & 23 & 24 & 25
            \end{bsmallmatrix}\) & 5\\
            
            
            \((\k^{\times})^5 \rtimes \langle (1\,2) \rangle\) & \(\begin{smallpmatrix}
                1 & -1 & a & b & c\\
                -1 & 1 & a & b & c\\
                a^{-1} & a^{-1} & 1 & d & e \\
                b^{-1} & b^{-1} & d^{-1} & 1 & f \\
                c^{-1} & c^{-1} & e^{-1} & f^{-1} & 1
            \end{smallpmatrix}\) & \(\langle (1\,2) \rangle\) & \(\begin{bsmallmatrix}
                1 & 2 & 3 & 4 & 5\\
                2 & 1 & 3 & 4 & 5\\
                6 & 6 & 7 & 8 & 9 \\
                10 & 10 & 11 & 12 & 13 \\
                14 & 14 & 15 & 16 & 17
            \end{bsmallmatrix}\) & 3\\
            
            
            \((\k^{\times})^5 \rtimes \langle (1\,2\,3) \rangle\) & \(\begin{smallpmatrix}
                1 & a & a^{-1} & c & d\\
                a^{-1} & 1 & a & c & d\\
                a & a^{-1} & 1 & c & d \\
                c^{-1} & c^{-1} & c^{-1} & 1 & f \\
                d^{-1} & d^{-1} & d^{-1} & f^{-1} & 1
            \end{smallpmatrix}\) & \(\langle (1\,2\,3) \rangle\) & \(\begin{bsmallmatrix}
                1 & 2 & 3 & 4 & 5\\
                3 & 1 & 2 & 4 & 5\\
                2 & 3 & 1 & 4 & 5 \\
                6 & 6 & 6 & 7 & 8 \\
                9 & 9 & 9 & 10 & 11
            \end{bsmallmatrix}\) & 4\\
            
            
            \((\k^{\times})^5 \rtimes \langle (1\,2\,3\,4) \rangle\) & \(\begin{smallpmatrix}
                1 & a & -1 & a^{-1} & d\\
                a^{-1} & 1 & a & -1 & d\\
                -1 & a^{-1} & 1 & a & d \\
                a & -1 & a^{-1} & 1 & d \\
                d^{-1} & d^{-1} & d^{-1} & d^{-1} & 1
            \end{smallpmatrix}\) & \(\langle (1\,2\,3\,4) \rangle\) & \(\begin{bsmallmatrix}
                1 & 2 & 3 & 4 & 5\\
                4 & 1 & 2 & 3 & 5\\
                3 & 4 & 1 & 2 & 5 \\
                2 & 3 & 4 & 1 & 5 \\
                6 & 6 & 6 & 6 & 7
            \end{bsmallmatrix}\) & 5\\
            
            
            \((\k^{\times})^5 \rtimes \langle (1\,2)(3\,4) \rangle\) & \(\begin{smallpmatrix}
                1 & \pm 1 & a & b & c\\
                \pm 1 & 1 & b & a & c\\
                a^{-1} & b^{-1} & 1 & \pm 1 & d \\
                b^{-1} & a^{-1} & \pm 1 & 1 & d \\
                c^{-1} & c^{-1} & d^{-1} & d^{-1} & 1
            \end{smallpmatrix}\) & \(\langle (1\,2)(3\,4) \rangle\) & \(\begin{bsmallmatrix}
                1 & 2 & 3 & 4 & 5\\
                2 & 1 & 4 & 3 & 5\\
                6 & 7 & 8 & 9 & 10 \\
                7 & 6 & 9 & 8 & 10 \\
                11 & 11 & 12 & 12 & 13
            \end{bsmallmatrix}\) & 3 \\
            
            
            \((\k^{\times})^5 \rtimes \langle (1\, 2),(3\, 4) \rangle\) & \(\begin{smallpmatrix}
                1 & \pm 1 & a & a & b\\
                \pm 1 & 1 & a & a & b\\
                a^{-1} & a^{-1} & 1 & \pm 1 & d \\
                a^{-1} & a^{-1} & \pm 1 & 1 & d \\
                b^{-1} & b^{-1} & d^{-1} & d^{-1} & 1
            \end{smallpmatrix}\) & \(\langle (1\,2),(3\,4) \rangle\) & \(\begin{bsmallmatrix}
                1 & 2 & 3 & 3 & 4\\
                2 & 1 & 3 & 3 & 4\\
                5 & 5 & 6 & 7 & 8 \\
                5 & 5 & 7 & 6 & 8 \\
                9 & 9 & 10 & 10 & 11
            \end{bsmallmatrix}\) & 3 \\
            
            
            \((\k^{\times})^5 \rtimes \langle (1\,2),(1\,2\,3) \rangle\) & \(\begin{smallpmatrix}
                1 & - 1 & -1 & a & b\\
                - 1 & 1 & -1 & a & b\\
                -1 & -1 & 1 & a & b \\
                a^{-1} & a^{-1} & a^{-1} & 1 & c \\
                b^{-1} & b^{-1} & b^{-1} & c^{-1} & 1
            \end{smallpmatrix}\) & \(\langle (1\,2),(1\,2\,3) \rangle\) & \(\begin{bsmallmatrix}
                1 & 2 & 2 & 3 & 4\\
                2 & 1 & 2 & 3 & 4\\
                2 & 2 & 1 & 3 & 4 \\
                5 & 5 & 5 & 6 & 7 \\
                8 & 8 & 8 & 9 & 10
            \end{bsmallmatrix}\) & 3\\
            
            
            \((\k^{\times})^5 \rtimes \langle (1\,2\,3\,4), (1\,3) \rangle\) & \(\begin{smallpmatrix}
                1 & 1 & - 1 & 1 & a\\
                1 & 1 & 1 & - 1 & a\\
                -1 & 1 & 1 & 1 & a \\
                1 & -1 & 1 & 1 & a \\
                a^{-1} & a^{-1} & a^{-1} & a^{-1} & 1
            \end{smallpmatrix}\) & \(\langle (1\,2\,3\,4), (1\,3) \rangle\) & \(\begin{bsmallmatrix}
                1 & 2 & 3 & 2 & 4\\
                2 & 1 & 2 & 3 & 4\\
                3 & 2 & 1 & 2 & 4 \\
                2 & 3 & 2 & 1 & 4 \\
                5 & 5 & 5 & 5 & 6
            \end{bsmallmatrix}\) & 3 \\
            
            
            \((\k^{\times})^5 \rtimes \langle (1\,2\,3\,4), (1\,2) \rangle\) & \(\begin{smallpmatrix}
                1 & -1 & -1 & -1 & a\\
                -1 & 1 & -1 & -1 & a\\
                -1 & -1 & 1 & -1 & a \\
                -1 & -1 & -1 & 1 & a \\
                a^{-1} & a^{-1} & a^{-1} & a^{-1} & 1
            \end{smallpmatrix}\) & \(\langle (1\,2\,3\,4), (1\,2) \rangle\) & \(\begin{bsmallmatrix}
                1 & 2 & 2 & 2 & 3\\
                2 & 1 & 2 & 2 & 3\\
                2 & 2 & 1 & 2 & 3 \\
                2 & 2 & 2 & 1 & 3 \\
                4 & 4 & 4 & 4 & 5
            \end{bsmallmatrix}\) & 3 \\
            
            
            \((\k^{\times})^5 \rtimes \langle (1\,2),(3\,4\,5) \rangle\) & \(\begin{smallpmatrix}
                1 & \pm 1 & a & a & a\\
                \pm 1 & 1 & a & a & a\\
                a^{-1} & a^{-1} & 1 & b & b^{-1} \\
                a^{-1} & a^{-1} & b^{-1} & 1 & b \\
                a^{-1} & a^{-1} & b & b^{-1} & 1
            \end{smallpmatrix}\) & \(\langle (1\,2),(3\,4\,5) \rangle\) & \(\begin{bsmallmatrix}
                1 & 2 & 3 & 3 & 3\\
                2 & 1 & 3 & 3 & 3\\
                4 & 4 & 5 & 6 & 7 \\
                4 & 4 & 7 & 5 & 6 \\
                4 & 4 & 6 & 7 & 5
            \end{bsmallmatrix}\) & 5 \\
            
            
            \((\k^{\times})^5 \rtimes \langle (1\,2),(1\,2\,3)(4\,5) \rangle\) & \(\begin{smallpmatrix}
                1 & - 1 & -1 & a & a\\
                - 1 & 1 & -1 & a & a\\
                -1 & -1 & 1 & a & a \\
                a^{-1} & a^{-1} & a^{-1} & 1 & -1 \\
                a^{-1} & a^{-1} & a^{-1} & -1 & 1
            \end{smallpmatrix}\) & \(\langle (1\,2),(1\,2\,3)(4\,5) \rangle\) & \(\begin{bsmallmatrix}
                1 & 2 & 2 & 3 & 3\\
                2 & 1 & 2 & 3 & 3\\
                2 & 2 & 1 & 3 & 3 \\
                4 & 4 & 4 & 5 & 6 \\
                4 & 4 & 4 & 6 & 5
            \end{bsmallmatrix}\) & 3\\
            
            
            \((\k^{\times})^5 \rtimes \langle (1\,2\,3\,4\,5) \rangle\) & \(\begin{smallpmatrix}
                1 & a & b & b^{-1} & a^{-1}\\
                a^{-1} & 1 & a & b & b^{-1}\\
                b^{-1} & a^{-1} & 1 & a & b \\
                b & b^{-1} & a^{-1} & 1 & a \\
                a & b & b^{-1} & a^{-1} & 1
            \end{smallpmatrix}\) & \(\langle (1\,2\,3\,4\,5) \rangle\) & \(\begin{bsmallmatrix}
                1 & 2 & 3 & 4 & 5\\
                5 & 1 & 2 & 3 & 4\\
                4 & 5 & 1 & 2 & 3 \\
                3 & 4 & 5 & 1 & 2 \\
                2 & 3 & 4 & 5 & 1
            \end{bsmallmatrix}\) & 4\\
            
            
            \((\k^{\times})^5 \rtimes \langle (1\,2\,3\,4\,5), (1\,5)(2\,4) \rangle\) & \(\begin{smallpmatrix}
                1 & -1 & 1 & 1 & -1\\
                -1 & 1 & -1 & 1 & 1\\
                1 & -1 & 1 & -1 & 1 \\
                1 & 1 & -1 & 1 & -1 \\
                -1 & 1 & 1 & -1 & 1
            \end{smallpmatrix}\) & \(\langle (1\,2\,3\,4\,5), (1\,5)(2\,4) \rangle\) & \(\begin{bsmallmatrix}
                1 & 2 & 3 & 3 & 2\\
                2 & 1 & 2 & 3 & 3\\
                3 & 2 & 1 & 2 & 3 \\
                3 & 3 & 2 & 1 & 2 \\
                2 & 3 & 3 & 2 & 1
            \end{bsmallmatrix}\) & 3\\
            
            
            \((\k^{\times})^5 \rtimes \mathfrak{S}_5\) & \(\begin{smallpmatrix}
                1 & -1 & -1 & -1 & -1\\
                -1 & 1 & -1 & -1 & -1\\
                -1 & -1 & 1 & -1 & -1 \\
                -1 & -1 & -1 & 1 & -1 \\
                -1 & -1 & -1 & -1 & 1
            \end{smallpmatrix}\) & \(
            \mathfrak{S}_5\) & \(\begin{bsmallmatrix}
                1 & 2 & 2 & 2 & 2\\
                2 & 1 & 2 & 2 & 2\\
                2 & 2 & 1 & 2 & 2 \\
                2 & 2 & 2 & 1 & 2 \\
                2 & 2 & 2 & 2 & 1
            \end{bsmallmatrix}\)\rule[-18pt]{0pt}{0pt} & 3 \\
            \hline
        \end{longtable}
    \endgroup
   \end{footnotesize}
   
    \newpage
    
    \begin{footnotesize}
    \begingroup
        \renewcommand\arraystretch{2.9}
        \centering
        \setlength{\tabcolsep}{2pt}
        \begin{longtable}{|ccccc|}
            \caption{Nonmonomial 
            Graded Automorphism Groups for
            $\dim V = 5$ } \\
            \hline
            \(\Aut_{\gr}(S_\q(V))\) & \(\q\) & \(\Stab \q\) & Orbits & Min \(|\k|\) \hspace{3pt} \\
            \hline
            \rule{0ex}{6.5ex}
            \((\k^{\times})^3 \times \GL(2,\k)\) & 
            \ \ \(\begin{smallpmatrix}
                1 & a & b & c & c \\
                a^{-1} & 1 & d & e & e \\
                b^{-1} & d^{-1} & 1 & f & f \\
                c^{-1} & e^{-1} & f^{-1} & 1 & 1 \\
                c^{-1} & e^{-1} & f^{-1} & 1 & 1
            \end{smallpmatrix}\) \ \ 
            & \(\langle e \rangle\) & \(\begin{bsmallmatrix}
                1 & 2 & 3 \\
                4 & 5 & 6 \\
                7 & 8 & 9 \\
            \end{bsmallmatrix}, \begin{bsmallmatrix}
                1
            \end{bsmallmatrix}\) & 4 \\
            
            
            \(( (\k^{\times})^3 \times \GL(2, \k)) \rtimes \langle (1\,2) \rangle\) &
            \ \(\begin{smallpmatrix}
                1 & -1 & a & b & b \\
                -1 & 1 & a & b & b \\
                a^{-1} & a^{-1} & 1 & c & c \\
                b^{-1} & b^{-1} & c^{-1} & 1 & 1 \\
                b^{-1} & b^{-1} & c^{-1} & 1 & 1
            \end{smallpmatrix}\) \ \ 
            & \(\langle (1\,2) \rangle\) & \(\begin{bsmallmatrix}
                1 & 2 & 3 \\
                2 & 1 & 3 \\
                4 & 4 & 5
            \end{bsmallmatrix}, \begin{bsmallmatrix}
                1
            \end{bsmallmatrix}\) & 5 \\
            
            
            \(( (\k^{\times})^3 \times \GL(2, \k)) \rtimes \langle (1\,2\,3) \rangle\) & \(\begin{smallpmatrix}
                1 & a & a^{-1} & b & b \\
                a^{-1} & 1 & a & b & b \\
                a & a^{-1} & 1 & b & b \\
                b^{-1} & b^{-1} & b^{-1} & 1 & 1 \\
                b^{-1} & b^{-1} & b^{-1} & 1 & 1
            \end{smallpmatrix}\) & \(\langle (1\,2\,3) \rangle\) & \(\begin{bsmallmatrix}
                1 & 2 & 3 \\
                3 & 1 & 2 \\
                2 & 3 & 1 \\
            \end{bsmallmatrix}, \begin{bsmallmatrix}
                1
            \end{bsmallmatrix}\) & 4 \\
            
            
            \(((\k^{\times})^3 \times \GL(2, \k)) \rtimes \mathfrak{S}_3\) & \(\begin{smallpmatrix}
                1 & -1 & -1 & a & a \\
                -1 & 1 & -1 & a & a \\
                -1 & -1 & 1 & a & a \\
                a^{-1} & a^{-1} & a^{-1} & 1 & 1 \\
                a^{-1} & a^{-1} & a^{-1} & 1 & 1
            \end{smallpmatrix}\) & \( \mathfrak{S}_3\) & \(\begin{bsmallmatrix}
                1 & 2 & 2 \\
                2 & 1 & 2 \\
                2 & 2 & 1 \\
            \end{bsmallmatrix}, \begin{bsmallmatrix}
                1
            \end{bsmallmatrix}\) & 3 \\
            
            
            \(\k^\times \times (\GL(2, \k))^2\) & \(\begin{smallpmatrix}
                1 & a & a & b & b \\
                a^{-1} & 1 & 1 & c & c\\
                a^{-1} & 1 & 1 & c & c\\
                b^{-1} & c^{-1} & c^{-1} & 1 & 1 \\
                b^{-1} & c^{-1} & c^{-1} & 1 & 1
            \end{smallpmatrix}\) & \(\langle e \rangle\) & \( \begin{bsmallmatrix}
                1
            \end{bsmallmatrix}, \begin{bsmallmatrix}
                1 & 2 \\
                3 & 4
            \end{bsmallmatrix}\) & 3\\
            
            
            \((\k^\times \times (\GL(2, \k))^2) \rtimes \mathfrak{S}_2\) & \(\begin{smallpmatrix}
                1 & a & a & b & b \\
                a^{-1} & 1 & 1 & -1 & -1\\
                a^{-1} & 1 & 1 & -1 & -1\\
                b^{-1} & -1 & -1 & 1 & 1 \\
                b^{-1} & -1 & -1 & 1 & 1
            \end{smallpmatrix}\) & \( \mathfrak{S}_2\) & \(\begin{bsmallmatrix}
                1
            \end{bsmallmatrix},
            \begin{bsmallmatrix}
                1 & 2 \\
                2 & 1
            \end{bsmallmatrix}\) & 3 \\
            
            
            \((\k^{\times})^2 \times \GL(3, \k)\) & \(\begin{smallpmatrix}
                1 & a & b & b & b \\
                a^{-1} & a & c & c & c \\
                b^{-1} & c^{-1} & 1 & 1 & 1 \\
                b^{-1} & c^{-1} & 1 & 1 & 1 \\
                b^{-1} & c^{-1} & 1 & 1 & 1
            \end{smallpmatrix}\) & \(\langle e \rangle\) & \(\begin{bsmallmatrix}
                1 & 2 \\
                3 & 4
            \end{bsmallmatrix}, \begin{bsmallmatrix}
                1
            \end{bsmallmatrix}\) & 3 \\
            
            
            \(((\k^{\times})^2 \times \GL(3, \k))
                \rtimes \langle (1\,2)\rangle\) & \(\begin{smallpmatrix}
                1 & -1 & a & a & a \\
                -1 & 1 & a & a & a \\
                a^{-1} & a^{-1} & 1 & 1 & 1 \\
                a^{-1} & a^{-1} & 1 & 1 & 1 \\
                a^{-1} & a^{-1} & 1 & 1 & 1 
            \end{smallpmatrix}\) & \(\langle (1\,2)\rangle\) & \(\begin{bsmallmatrix}
                1 & 2 \\
                2 & 1
            \end{bsmallmatrix}, \begin{bsmallmatrix}
                1
            \end{bsmallmatrix}\) & 3 \\
            
            
            \(\begin{matrix}
                \GL(2, \k) \times
                \GL(3, \k)
            \end{matrix}\) & \(\begin{smallpmatrix}
                1 & 1 & a & a & a \\
                1 & 1 & a & a & a \\
                a^{-1} & a^{-1} & 1 & 1 & 1 \\
                a^{-1} & a^{-1} & 1 & 1 & 1 \\
                a^{-1} & a^{-1} & 1 & 1 & 1
            \end{smallpmatrix}\) & \(\langle e \rangle\) & \(\begin{bsmallmatrix}
                1 \\
            \end{bsmallmatrix}, \begin{bsmallmatrix}
                1 \\
            \end{bsmallmatrix}\) & 3\\
            
            
            \( \k^\times \times \GL(4, \k)\) & \(\begin{smallpmatrix}
                1 & a & a & a & a \\
                a^{-1} & 1 & 1 & 1 & 1 \\
                a^{-1} & 1 & 1 & 1 & 1 \\
                a^{-1} & 1 & 1 & 1 & 1 \\
                a^{-1} & 1 & 1 & 1 & 1
            \end{smallpmatrix}\) & \(\langle e \rangle \times \langle e \rangle\) & \(\begin{bsmallmatrix}
                1
            \end{bsmallmatrix}, \begin{bsmallmatrix}
                1
            \end{bsmallmatrix}\) & 3 \\
            
            
            \(\begin{matrix}
                \GL(5, \k)
            \end{matrix}\) & \(\begin{smallpmatrix}
                1 & 1 & 1 & 1 & 1\\
                1 & 1 & 1 & 1 & 1\\
                1 & 1 & 1 & 1 & 1 \\
                1 & 1 & 1 & 1 & 1 \\
                 1 & 1 & 1 & 1 & 1
            \end{smallpmatrix}\) & \(\langle e \rangle\) & \(\begin{bsmallmatrix}
                1
            \end{bsmallmatrix}\)\rule[-18pt]{0pt}{0pt} & 2 \\
            \hline
        \end{longtable}
    \endgroup
       \end{footnotesize}

\vspace{0ex}

\subsection*{Dimension six}
We now assume $\dim V =6$.
There are total of 65
graded automorphism groups.
There are 11 partitions of \(6\),
with the partition \( (1\ 1 \ 1 \ 1 \ 1 \ 1) \) corresponding to monomial automorphisms.
The next table provides the number
of groups
        corresponding to the fixed block decomposition
        given by each partition.
        We give the monomial groups explicitly
        below.
\vspace{-0.2cm}
\begin{table}[h!]
    \centering
    \caption{Graded Automorphism Groups for \(\dim V = 6\) organized by partition}
    \begin{tabular}{|c|c|c|c|c|c|c|c|c|c|c|c|}
        \hline
        $\lambda$ & $(1^6)$ & $(1^4\,2)$ & $(1^2\,2^2)$ & $(2^3)$ & $(1^3\,3)$ & $(1\,2\,3)$ & $(3^2)$ & $(1^2\,4)$ & $(2\,4)$ & $(1\,5)$ & $(6)$\rule{0ex}{2.5ex} \\ 
        \hline
        count & 36 & 9 & 4 & 4 & 4 & 1 & 2 & 2 & 1 & 1 & 1 \\
        \hline
    \end{tabular}
    \vspace{-0.5cm}
\end{table}
\newline
For $|\k|=3$
there are
$21$ monomial
automorphism groups; they have 
stabilizing permutation groups
\begin{align*}
    & \langle e\rangle,\ 
    \langle(1\, 2)\rangle,\ 
    \langle(1\, 2)(3\, 4)(5\, 6)\rangle,\ 
    \langle(1\, 2)(3\, 4)\rangle,\ 
    \langle(1\, 2)(3\, 4), (5\, 6)\rangle,\ 
    \langle(1\, 2), (3\, 4)\rangle,\ \langle(1\, 5)(3\, 4), (5\, 6)(2\, 4)\rangle, \\
    & \langle(1\, 3)(2\, 4), (1\, 2)(3\, 4)(5\, 6)\rangle,\ 
    \langle(1\, 2), (2\, 3)\rangle,\ 
    \langle(1\, 2\, 3\, 4)(5\, 6), (2\, 4)\rangle,\ \langle(1\, 2\, 3\, 4), (2\, 4)\rangle,\ \langle(1\, 2\, 3), (3\, 4)\rangle, \\
    &
    \langle(1\, 5)(2\, 4), (1\, 2\, 3\, 4\, 5)\rangle,\ \langle(1\, 2), (2\, 3), (4\, 5)\rangle,\ \langle(1\, 2\, 3\, 4), (4\, 5)\rangle,\ \langle(1\, 2\, 3\, 4\, 5\, 6), (1\, 3)\rangle,\,
    \langle(1\, 2), (2\, 3\, 4)(5\, 6)\rangle, \\
    & 
    \langle(1\, 2\, 3\, 4\, 5\, 6), (1\, 6\, 3\, 4\, 5\, 2)\rangle,\ \langle(1\, 2\, 3\, 4\, 5\, 6), (1\, 3\, 4\, 6), (2\, 5)\rangle,\ \langle(1\, 2\, 3\, 4), (1\, 3), (5\, 6)\rangle,\
    \langle(1\, 2\, 3\, 4\, 5), (5\, 6)\rangle
    .
\end{align*}
For 
$|\k| \geq 4$ and 
$\text{char } \k=2$,
there are $12$ monomial 
automorphism groups: 
$4$ groups arise from  the \(|\k| = 3\) case
and have stabilizing permutation groups
$$
\langle e\rangle,\ \langle(1\, 2)(3\, 4)\rangle,\ \langle(1\, 2)(3\, 4)(5\, 6)\rangle,\ 
\langle(1\, 2\, 3)(4\, 5\, 6), (1\, 2)(4\, 5) \rangle,
$$
and $8$ additional  groups
have stabilizing permutation groups
\begin{align*}
    & \langle(1\, 2\, 3)\rangle,\
\langle(1\, 2\, 3)(4\, 5\, 6)\rangle,\
\langle(1\, 2\, 3\, 4)(5\, 6)\rangle,\
\langle(1\, 2\, 3\, 4)\rangle, \\
& \langle(1\, 2\, 3\, 4\, 5)\rangle,\
\langle(1\, 2\, 3\, 4\, 5\, 6)\rangle,\
\langle(1\, 2\, 3), (4\, 5\, 6)\rangle,\
\langle(1\, 2)(3\, 4)(5\, 6), (1\, 3\, 5)\rangle
\, .
\end{align*}
\noindent
For \(|\k| \geq 5\) and 
 $\text{char } \k \neq 2$,
there are $36$
monomial automorphism groups:
$21+8=29$ with stabilizing permutation groups
as in the cases \(|\k| = 3\)
and \(|\k| = 4\) 
and $7$ with stabilizing permutation groups
\begin{align*}
    & \langle(1\, 2\, 3)(4\, 5\, 6), (3\, 4)(2\, 5)(1\, 6)\rangle,\ 
    \langle(1\, 2\, 3)(4\, 5)\rangle,\
    \langle(1\, 2), (3\, 4), (5\, 6)\rangle,\ 
    \langle(1\, 2), (3\, 4\, 5\, 6)\rangle, \\
    & \langle(1\, 2\, 3), (4\, 5\, 6), (4\, 5)\rangle,\ \langle(1\, 2\, 3\, 4\, 5\, 6), (1\, 4)\rangle,\
    \langle(2\, 3), (1\, 2\, 3)(4\, 5\, 6), (5\, 6)\rangle
    \, .
\end{align*}

\subsection*{Dimension seven}
We now assume $\dim V =7$. There are a total of 105
graded automorphism groups.
These arise from the $15$ partitions of \(7\),
with the partition \((1\ 1 \ 1 \ 1 \ 1 \ 1 \ 1) \) corresponding to monomial automorphism groups.
The next two tables provides a count of the number
of graded automorphism groups corresponding to the fixed block decomposition given by each partition.

\begin{table}[h!]
    \centering
    \caption{Graded Automorphism Groups for \(\dim V = 7\) organized by partition}
    \begin{tabular}{|c|c|c|c|c|c|c|c|c|c|c|c|}
        \hline
        $\lambda$ & $(1^7)$ & $(1^5\,2)$ & $(1^3\,2^2)$ & $(1\,2^3)$ & $(1^4\,3)$ & $(1^2\,2\,3)$ & $(2^2\,3)$ & $(1\,3^2)$\rule{0ex}{2.5ex} \\ 
        \hline
        count & 53 & 14 & 8 & 4 & 9 & 2 & 2 & 2 \\
        \hline
    \end{tabular}
  \rule[-3ex]{0ex}{8ex}
    \begin{tabular}{|c|c|c|c|c|c|c|c|c|c|c|c|}
        \hline
        $\lambda$ & $(1^3\,4)$ & $(1\,2\,4)$ & $(3\,4)$ & $(1^2\,5)$ & $(2\,5)$ & $(1\, 6)$ & $(7)$\rule{0ex}{2.5ex} \\ 
        \hline
        count & 4 & 1 & 1 & 2 & 1 & 1 & 1 \\
        \hline
    \end{tabular}
\end{table}
\noindent
For \(|\k| = 3\),\! there are $31$ 
monomial automorphism groups;\! they have 
stabilizing permutation groups 
\begin{align*}
    & \langle e\rangle,\ 
    \langle(1\, 2)\rangle,\ 
    \langle(1\, 2)(3\, 4)\rangle,\ 
    \langle(1\, 2)(3\, 4)(5\, 6)\rangle,\ \langle(1\, 2), (3\, 4)\rangle,\ 
    \langle(1\, 3)(2\, 4), (1\, 2)(3\, 4)(5\, 6)\rangle,\ \langle(1\, 2), (2\, 3)\rangle,
    \\
    & \langle(1\, 2)(3\, 4), (5\, 6)\rangle,\  \langle(1\, 2)(4\, 5), (2\, 3)(5\, 6)\rangle,\ 
    \langle(1\, 3), (1\, 2\, 3\, 4)\rangle,\
    \langle(1\, 2), (3\, 4), (5\, 6)\rangle,\ \langle(1\, 3), (1\, 2\, 3\, 4)(5\, 6)\rangle,
    \\
    &
    \langle(1\, 2)(3\, 4), (1\, 4)(5\, 2)\rangle,\ 
    \langle(1\, 2), (2\, 3), (4\, 5)(6\, 7)\rangle,\ 
    \langle(1\, 2), (2\, 3), (4\, 5)\rangle,\  \langle(1\, 6)(3\, 5), (1\, 2)(3\, 4)(5\, 6)\rangle,
    \\
    &
    \langle(1\, 2\, 3\, 4\, 5\, 6\, 7), (1\, 7)(2\, 6)(3\, 5)\rangle,\ 
    \langle(1\, 3), (1\, 2\, 3\, 4), (5\, 6)\rangle,\ \langle(1\, 2)(3\, 5), (1\, 2\, 3\, 4\, 5), (6\, 7)\rangle,\ 
    \langle(1\, 2\, 3), (3\, 4)\rangle,
    \\
    &
    \langle(1\, 2), (3\,4) (5\,6) (5\,7)\rangle,\ \langle(1\, 2), (1\, 2\, 3\, 4\, 5)\rangle,\ \langle(1\,2), (1\, 3), (4\, 5), (4\, 6)\rangle,\ 
    \langle(1\, 2\, 3\, 4), (1\, 4), (5\, 6)\rangle,
    \\
    &
    \langle(1\, 2)(3\, 4), (1\, 5\, 4\, 6)(2\, 3), (1\, 4)\rangle,\ \langle(1\, 3), (1\, 2\, 3\, 4), (5\, 6), (6\, 7)\rangle ,\ 
    \langle(1\, 2)(3\, 4)(5\, 6), (2\, 3), (4\, 5)\rangle,
    \\
    &
    \langle(3\, 4\, 5\, 6)(1\, 7), (1\, 2)(3\, 4\, 5)\rangle,\ \langle(1\, 2\, 3)(4\, 5)(6\, 7), (1\, 2\, 4)(3\, 5)\rangle,\ 
\langle(1\, 2\, 3\, 4\, 5), (5\, 6)\rangle,\ 
\langle(1\, 2\, 3\, 4\, 5\, 6), (6\, 7)\rangle
\, . 
\end{align*}
For $|\k|=4$,
there are $16$ 
monomial automorphism groups:
$4$ 
from the \(|\k| = 3\) case, namely,
those with stabilizing permutation groups
$$\langle e\rangle,\ \langle(1\, 2)(3\, 4)\rangle,\ \langle(1\, 2)(3\, 4)(5\, 6)\rangle,\ \langle(1\, 2\, 3)(4\, 5\, 6), (1\, 2)(4\, 5)\rangle,
$$
and $12$ additional 
groups
with stabilizing permutation groups
\begin{align*}
   & \langle(1\, 2\, 3)\rangle,\ 
\langle(1\, 2\, 3)(4\, 5\, 6)\rangle,\ 
\langle(1\, 2\, 3\, 4)\rangle,\ 
\langle(1\, 2\, 3\, 4)(5\, 6)\rangle,\ 
\langle(1\, 2\, 3\, 4\, 5)\rangle, \
\langle(1\, 2\, 3)(4\, 5)(6\, 7)\rangle,\
\langle(1\, 2\, 3\, 4\, 5\, 6)\rangle,
\\
 & 
\langle(1\, 2\, 3\, 4\, 5\, 6\, 7)\rangle,\
\langle(1\, 2\, 3), (4\, 5\, 6)\rangle,\ \langle(1\, 2\, 3\, 4)(5\, 6\, 7)\rangle,\  \langle(1\, 3\, 5), (1\, 2\, 3\, 4\, 5\, 6)\rangle,\ 
\langle(1\, 2\, 3\, 4\, 5\, 6\, 7), (2\, 3\, 5)(4\, 7\, 6)\rangle
\, .
\end{align*}
For
\(|\k| \geq 5\),
there are $53$
monomial
automorphism groups:
$31+12=43$ with stabilizing permutation groups
as in the cases \(|\k| = 3\)
and \(|\k| = 4\) 
above 
and $10$ additional groups
with
types
\begin{align*}
& \langle(1\, 6)(2\, 5)(3\, 4), (1\, 2\, 3\, 4\, 5\, 6)\rangle,\
\langle(1\, 3\, 5)(2\, 4)\rangle,\
\langle(1\, 2), (3\, 4\, 5\, 6)\rangle,\ 
\langle(1\, 2\, 3\, 4\, 5)(6\, 7)\rangle,
\\
& \langle(1\, 2), (3\, 4\, 5), (6\, 7)\rangle,\ 
\langle(1\, 2), (2\, 3), (4\, 5\, 6)\rangle,\ 
\langle(1\, 2\, 3)(4\, 5\, 6), (1\, 4)\rangle,
\\
& \langle(1\, 3), (1\, 2\, 3\, 4), (5\, 6\, 7)\rangle,\ 
\langle(1\, 2\, 3\, 4), (5\, 6), (6\, 7)\rangle,\
\langle(1\, 2\, 3), (4\, 5\, 6), (6\, 7)\rangle
\, .
\end{align*}



\vspace{1ex}

\section{Acknowledgments}
The authors thank James J. Zhang for generously
sharing his expertise and for
helpful discussions on this topic.


\vspace{1ex}

\section{Declarations}

The authors declare that the data supporting the 
findings of this 
article are available within the paper.
The authors have no relevant financial or non-financial interests to
disclose, nor any relevant conflicts of interest to disclose.
No research for the article involved human or animal participants.
The second author was partially supported by Simons grants 
  429539 and 949953.

\end{document}